\documentclass[12pt]{article}

\usepackage[british]{babel}
\usepackage[babel]{csquotes}
\usepackage{amsmath}
\usepackage{amssymb}
\usepackage{amsthm}     % for theoremstyles
\usepackage{amsfonts}   % for \mathbb{E}
\usepackage{mathtools}  % for := \coloneqq
\usepackage{mathrsfs}   % for \mathsrc{S}
\usepackage{bbm}        % for \mathbbm{1} (schöne indikatorfunktion)
\usepackage{ragged2e}	% für \Centering und \justifying (zentriert text sodass links und rechts alles auf gleicher Höhe)
\usepackage{dsfont}

\usepackage{graphicx}
\usepackage{caption}
\usepackage{subcaption}
\usepackage{float}

\usepackage{tikz}
\usepackage{pgfplots} 
\usepackage{lscape} 

\usepackage{hyperref}
\usepackage{enumitem}
\usepackage{ulem}
\usepackage{url}
\usepackage{eurosym}

\newtheorem{Definition}{Definition}[section]

\newtheorem{Theorem}[Definition]{Theorem}
\newtheorem{Lemma}[Definition]{Lemma}
\newtheorem{Proposition}[Definition]{Proposition}

\newtheorem{Corollary}[Definition]{Corollary}

\newcommand{\1}{\mathds{1}}
\newcommand{\Cov}{\text{Cov}}
\DeclareMathOperator*{\esssup}{ess\,sup}

\def\p{\mathbb{P}}								% nur das P; \def statt \newcommand, da \P ansonsten schon definiert
\def\E{\mathbb{E}}                              % nur das E
\def\V{\mathbb{V}}                              % nur das V

\newcommand\R{\mathbb{R}}
\newcommand\N{\mathbb{N}}

\newcommand\eps{\varepsilon}

\definecolor{mygreen}{RGB}{3,184,0}
\definecolor{mylila}{RGB}{186,0,183}

\newcommand\pt[1]{\text{[#1 \ifnum #1=1{point}\else{points}\fi]}}

\let\BFseries\bfseries\def\bfseries{\BFseries\mathversion{bold}}

\pgfplotsset{compat=1.18}

\usepackage{pifont}% Zapf-Dingbats Symbole

\RequirePackage{amsthm,amsmath,amsfonts,mathrsfs,amssymb,hyperref}
\usepackage{xcolor}

\DeclareMathAlphabet{\mathdutchcal}{U}{dutchcal}{m}{n}
\SetMathAlphabet{\mathdutchcal}{bold}{U}{dutchcal}{b}{n}
\DeclareMathAlphabet{\mathdutchbcal}{U}{dutchcal}{b}{n}
\def\dd{\mbox{d}}
\def\E{\mathbb{E}}
\def\V{\mathbb{V}}
\def\P{\mathbb{P}}
\def\R{\mathbb{R}}
\def\eps{\varepsilon}

\def\deq{\mathrel{\stackrel{d}{=}}} % equality in distribution
\usepackage{enumitem}
\setenumerate[1]{label=(\textup{\alph*})}
\setenumerate[2]{label=(\textup{\roman*})}
\def\dd{\mbox{d}}
\title{Persistence Probability of Fractional Brownian Motion with Random Hurst Exponent}
\author{
Frank Aurzada and Sabine Müller\thanks{Corresponding author: \texttt{sabine.mueller@tu-darmstadt.de}}\\
Technical University of Darmstadt, Darmstadt, Germany
}
\date{\today}

\begin{document}

\maketitle

\begin{abstract}
We study the persistence properties of a fractional Brownian motion
whose Hurst exponent is a random variable instead of a fixed constant. For each fixed \(H \in (0,1)\), it is well 
known that the persistence probability of an FBM below a constant barrier 
decays like \(T^{-(1-H)+o(1)}\), as $T$ tends to infinity, cf.\ \cite{molchan99}. Our object of interest is the 
persistence probability of the process resulting from first randomly selecting $H\in (0,1)$ and then considering a fractional Brownian motion with this value of $H$ as a Hurst exponent, a process that is referred to as a fractional Brownian motion with random exponent.
We prove that its persistence probability decays as $T^{-(1-H_0)+o(1)}$, as $T$ tends to infinity, where $H_0$ is the essential supremum of the distribution of the random Hurst exponent.
\end{abstract}

\allowdisplaybreaks

%%%%%%%%%%%%%%%%%%%%%%%
%%%%%%%%%%%%%%%%%%%%%%%
\section{Introduction and Main Result}
%%%%%%%%%%%%%%%%%%%%%%%
%%%%%%%%%%%%%%%%%%%%%%%

Fractional Brownian motion (FBM), first introduced by Kolmogorov \cite{kolmogorov1940_wienersche_spiralen} and popularized by Mandelbrot and Van Ness \cite{MandelbrotVanNess1968FBM}, is a centered Gaussian process $(B_t^H)_{t \geq 0}$ with covariance function
\begin{align*}
    \E\left[B_t^H B_s^H \right] = \frac{1}{2}\left( t^{2H}+s^{2H} - |t-s|^{2H}\right), \quad \text{$s, t \geq 0$,}
\end{align*}
where $H\in (0,1)$ is the Hurst exponent. The process is $H$-self-similar and has stationary increments and therefore is widely used as a  model for anomalous diffusion \cite{metzler2000random,MetzlerJeonCherstvyBarkai2014,RossoZoia2014FirstPassageAnomalousDiffusion}. 
However, recent single-particle tracking experiments in biological cells revealed highly complicated anomalous diffusion phenomena that cannot be adequately captured by a single self-similar model with fixed exponent \cite{BenelliWeiss2021NJP, SpecknerWeiss2021Entropy, SadoonWang2018PRE, CherstvyThapaWagnerMetzler2019MucinHydrogels}. 
This has motivated the study of extensions of classical FBM in which the Hurst exponent is allowed to vary, leading naturally to models with a random Hurst exponent. 

In this paper, we study the persistence probability of the extension introduced and systematically analysed in \cite{balcerek2022_FBMre_chaos}, namely the so-called FBMRE, fractional Brownian motion with random exponent. 
Subsequently, a number of works have explored further aspects of this model and closely related random exponent constructions, including statistical procedures to distinguish constant-$H$ FBM from random-$H$ behavior \cite{grzesiek2024_distinguishing_chaos}, extensions to Riemann--Liouville type processes \cite{woszczek2025_rlFBMre_chaos}, and dynamical or limit-theorem perspectives leading to FBM-type limits with (possibly time-dependent) random Hurst exponents \cite{bender2024_limit_theorem_arxiv}. 
Our contribution is to determine the annealed persistence asymptotics for this model; in particular, we prove that the persistence probability decays as \(T^{-(1-H_0)+o(1)}\), where \(H_0\) is the essential supremum of the random Hurst exponent.

The definition of the FBMRE is as follows. Let $\{(B_t^{H})_{t\ge 0}: H\in(0,1)\}$ be a family of fractional Brownian motions defined on a common probability space and let $\mathcal H$ be an $(0,1)$-valued random variable independent of this family. We write $\mathbb P_B$ (and sometimes only $\mathbb P$ if there is no ambiguity) for the probability measure underlying the family of fractional Brownian motions, and we denote the probability measure underlying $\mathcal H$ by $\p_\mathcal{H}$. Fractional Brownian motion with random exponent (FBMRE) is defined on the product space with product measure $\mathbb P_{\mathcal{H}}\otimes \mathbb P_B$ by simply plugging in $\mathcal{H}$ as Hurst exponent into the family $B^H$: $(B_t^\mathcal{H})_{t\geq 0}$.
In other words, conditionally on $\mathcal{H}=H$ the process is an FBM with Hurst exponent $H$, so that each single trajectory resembles a classical FBM path. Unconditionally, however, $B^\mathcal{H}$ is a mixture over Hurst exponents and in general no longer an FBM; in particular, it is neither self-similar in the usual sense nor a Gaussian process.

The main purpose of this work is to study the persistence probability for FBMRE. Persistence concerns the events when a stochastic process stays below a prescribed level for a long time, or equivalently that it exhibits an unusually long excursion without crossing a barrier. 
More generally, for a stochastic process $(Z_t)_{t\in I}$, where $I\subseteq \R$ is some index set, and a level $a\in\mathbb R$, one is interested in the decay rate of the probability
\[
\mathbb P\!\left( \sup_{t\in[0,T]\cap I} Z_t < a \right), \quad \text{as $T\to\infty$}.
\]
In many important examples, these probabilities decay polynomially, i.e.\ as $T^{-\theta +o(1)}$, and it is the first goal to find the persistence exponent $\theta$. For the relevance of this question in theoretical physics we recommend the monograph \cite{metzler2014firstpassage} and the surveys \cite{BrayMajumdarSchehr2013Persistence}, \cite{Majumdar1999PersistenceNoneq}. For an overview of the mathematical literature, we refer to the survey \cite{AurzadaSimon2015Persistence}.

From a physical point of view, persistence probabilities describe survival or non-crossing events, for instance the probability that a fluctuating interface, a tracer particle, or a transport coordinate has not yet reached a prescribed threshold or absorbing boundary. Such questions are particularly relevant for heterogeneous anomalous diffusions, where different particles, spatial regions, or experimental realisations may exhibit different effective Hurst exponents.
In this situation, the random exponent should be interpreted as a simple trajectory-wise model for quenched heterogeneity. The annealed persistence probability then corresponds to the survival probability observed at the ensemble level, when the individual value of the Hurst exponent is not conditioned on. Our result shows that the long-time decay of this ensemble survival probability is not governed by a typical or averaged Hurst exponent, but by the largest admissible value. Thus the slowest-decaying, most persistent subpopulation controls the asymptotic persistence behaviour,
which can be relevant for interpreting long-time first-passage data in heterogeneous media such as crowded biological environments, viscoelastic materials, or other systems displaying trajectory-to-trajectory variability in anomalous diffusion.

Before introducing the main result for FBM with random Hurst parameter, let us review the literature on fixed Hurst parameter FBM.  For the classical FBM with fixed $H\in(0,1)$, the persistence problem under a constant barrier is by now well understood at the level of the exponent.
In particular, Molchan \cite{molchan99} (see \cite{Aurzada:2011:OSE,Aurzada18,Peng2023FractionalBrownianMotion} for refinements) proved that
\begin{equation} \label{eqn:molchanT}
\mathbb P\!\left(\max_{t\in[0,T]} B_t^H < 1\right)
= T^{-(1-H)+o(1)},\qquad T\to\infty.
\end{equation}
For physical applications where the FBM persistence exponent $\theta=1-H$ appears naturally, see \cite{KrugKallabisMajumdarCornellBraySire1997,Majumdar2003MatheronDeMarsily,ZoiaRossoMajumdar2009,MajumdarRossoZoia2010}.
We take \eqref{eqn:molchanT} as the benchmark for our setting with random Hurst exponent. The object of interest in this paper is the annealed persistence probability
\begin{align*}
    \E_\mathcal{H}\left[ \p_B\left( \max_{t\in [0,T]}B_t^\mathcal{H}<1\right)\right],
\end{align*}
where $(B_t^\mathcal{H})_{t\geq 0}$ is an FBMRE. Before we can formulate the main result, we need to fix one more notation. For an FBMRE $(B_t^\mathcal{H})_{t\geq 0}$, we define
\begin{align*}
    H_0:= \esssup \mathcal{H},
\end{align*}
i.e.\ $\mathcal{H}\le H_0$ almost surely and for every $\varepsilon>0$,
$\P_\mathcal{H}(\mathcal{H} > H_0-\varepsilon) > 0$.

With that notation in place, we are ready to state our main result.

\begin{Theorem}\label{Thm}
    Let $(B_t^\mathcal{H})_{t\geq 0}$ be an FBMRE. Then % The asymptotic behavior of its persistence probability is given by
    \begin{align*}
        \E_\mathcal{H}\left[\p_B\left(\max_{t\in[0,T]}B_t^\mathcal{H} < 1\right)\right] = T^{-(1-H_0)+o(1)}, \quad \text{as $T \rightarrow \infty$}.
    \end{align*}
\end{Theorem}

The appearance of \(H_0\) in the theorem above has a simple heuristic interpretation. For a fixed Hurst exponent \(H\), Molchan's result \eqref{eqn:molchanT} says that the persistence probability decays like \(T^{-(1-H)+o(1)}\). Hence larger values of \(H\) correspond to smaller persistence exponents and therefore to slower decay. In the annealed average over the random Hurst exponent, one may therefore expect the long-time behaviour to be dominated by trajectories with Hurst exponent close to the largest possible value \(H_0\). This heuristic leads precisely to the exponent \(1-H_0\) in Theorem~\ref{Thm}. However, this reasoning cannot be turned into a proof by simply averaging the fixed-\(H\) asymptotics in \eqref{eqn:molchanT}, since the corresponding \(o(1)\)-terms degenerate as \(H\downarrow 0\).

For an FBM with fixed $H$, the persistence result in (\ref{eqn:molchanT}) is equivalent to 
\begin{equation} \label{eqn:molchaneps}
\mathbb P\!\left(\max_{t\in[0,1]} B_t^H < \eps\right)
= \eps^{\frac{(1-H)}{H}+o(1)},\qquad \eps\to 0,
\end{equation}
which follows directly from the self-similarity of FBM. However, we stress that FBMRE is in general not self-similar.
Nonetheless, it is natural to consider the related small-barrier problem, where the time interval is kept fixed and the barrier is sent to zero. For fixed $H$, the exponent \((1-H)/H\) is decreasing in $H$, so that larger Hurst exponents again correspond to slower decay. Thus, also in the small-barrier regime, one expects the annealed asymptotics to be governed by values of the random Hurst exponent close to \(H_0\). The result in this case is as follows.

\begin{Theorem} \label{thm:smallbarrier}
    Let $(B_t^\mathcal{H})_{t\geq 0}$ be an FBMRE. Then
    \begin{align}
        \E_\mathcal{H}\left[\p_B\left(\max_{t\in[0,1]}B_t^\mathcal{H} < \varepsilon\right)\right] = \varepsilon^{\frac{1-H_0}{H_0}+o(1)}, \quad \text{as $\varepsilon \rightarrow 0$.} \label{eqn:Anew}
    \end{align}
\end{Theorem}

%%%%%%%%%%%%%%%%%%%%%%%%%%
%%%%%%%%%%%%%%%%%%%%%%%%%%
%%%%%%%%%%%%%%%%%%%%%%%%%%

We shall see that Theorem~\ref{thm:smallbarrier} is in fact easy to prove; essentially it follows from (\ref{eqn:molchaneps}) and an application of Slepian's lemma. To the contrary, the proof of Theorem~\ref{Thm} is surprisingly involved. We stress again that Theorems~\ref{Thm} and~\ref{thm:smallbarrier} are not directly related due to the lack of self-similarity of FBMRE. 

Sometimes, one allows the Hurst exponent of an FBM to be equal to $1$, in which case one defines $B_t^1:=t\xi$, $t\ge0$, where $\xi\sim\mathcal N(0,1)$, see e.g.\ \cite{Cheridito2001MixedFBM}. Accordingly, one may extend the notion of an FBMRE to the case where $\mathcal H$ (still independent of the family $\{(B_t^H)_{t\ge0}:H\in(0,1]\}$) has an atom at $1$. It is easy to see that $\P\Big(\max_{t\in[0,T]}B_t^1<1\Big)\to 1/2$, so that Theorem~\ref{Thm} and Theorem~\ref{thm:smallbarrier} readily extend to the case where $\mathcal{H}$ has mass at $1$.

Let us discuss some open problems and generalisations.

\paragraph{Monotonicity in $H$.}  The approach 
 to estimate the small barrier probability (\ref{eqn:Anew}) works relatively easily because for the occurring probability, the function
\begin{align*}
    H\mapsto \p\left( \max_{t\in [0,1]}B_t^H < \varepsilon\right)
\end{align*}
is non-decreasing (cf.\ Lemma~\ref{lem:SlepianZeug} below). One is tempted to conjecture that a similar result holds for the function
\begin{align} \label{eqn:monotonicityquestion}
    H \mapsto \p\left( \max_{t\in [0,T]}B_t^H < 1\right),
\end{align}
enabling to proceed in the proof of Theorem~\ref{Thm} exactly as in the proof of Theorem~\ref{thm:smallbarrier}. 
However, as shown in the recent work \cite{tang_personal_communication}, the function in \eqref{eqn:monotonicityquestion} is not monotone in \(H\) for any \(T>1\). 

\paragraph{Multi-parameter FBM with random Hurst exponent.} Two generalizations of fractional Brownian motion with multi-dimensional index sets, $[0,\infty)^d$ or $\R^d$, are the fractional Brownian sheet, on the one hand, and L\'evy's fractional Brownian motion, on the other hand. For the latter, the main result for fixed Hurst exponent is given in Theorem~3 of \cite{molchan99}, the asymptotics is polynomial, as one would expect. One could also study this with randomized Hurst exponent. On the other hand, the fractional Brownian sheet can be defined in an anisotropic way with Hurst exponents $(H_1,\ldots,H_d)$. Similarly to the present paper, one can sample the Hurst exponents $\mathcal{H}=(\mathcal{H}_1,\ldots,\mathcal{H}_d)$ randomly with -- this time -- a distribution $\mathbb P_\mathcal{H}$ on $(0,1)^d$. The results we are aware of for deterministic Hurst exponents show an interesting behaviour in the sense that the persistence probability is of order $e^{-\theta (\ln T)^d (1+o(1))}$ instead of $T^{-\theta+o(1)}$. The main references are \cite{LiShao2004LowerTailGaussianProcesses,Molchan2008UnilateralSmallDeviations,Molchan2022PersistenceLamperti}. The question for the fractional Brownian sheet seems very interesting since $\mathcal{H}$ is then a $d$-dimensional random variable. We expect that the techniques from the present paper can be of use for this problem.

\paragraph{Other Generalizations.} As mentioned above, many generalizations of FBM with non-constant Hurst exponent exist. The persistence probabilities are not studied for any of them, be it in the form of Theorem~\ref{Thm} or Theorem~\ref{thm:smallbarrier}. Therefore, there are many possible generalizations of the present results and techniques. Let us mention some of the related work: Beyond randomizing a single trajectory-level exponent -- meaning that we sample the Hurst exponent once and then keep it fixed along the entire sample path -- there is a broad literature on models with time-varying local regularity, most prominently multifractional Brownian motion (MBM). We refer to \cite{LevyVehelPeltier1995mBm} for the original definition and to \cite{StoevTaqqu2006HowRichMBM} for structural results and also to works that incorporate additional randomness into the local exponent, see e.g.\ \cite{AyacheEsserHamonier2018MPRE,loboda2021_regular_mfare_spa}. The idea of allowing the local regularity exponent to be random and possibly even time-dependent appears already in \cite{PapanicolaouSolna2003WaveletKolmogorov} and was developed more systematically in the framework of multifractional processes with random exponent in \cite{AyacheTaqqu2005MPRE}. For further developments on multifractional random-exponent models see e.g.\ \cite{AyacheBouly2022MovingAverageMPRE,AyacheBouly2023UniformStrongConsistent,Loosveldt2023MultifractionalHermite}. \\The FBMRE studied here can be viewed as a simpler random-exponent mechanism, namely a trajectory-wise randomization of the Hurst exponent, where the exponent is sampled once and then kept constant along each path. The concept of random parameters is related to the framework of so-called su\-per\-sta\-tis\-tics, see \cite{BeckCohen2003Superstatistics,Beck2006SuperstatisticalBrownianMotion}. 
Complementary FBM-based models with different random-parameter scenarios have recently been explored in \cite{HanKorabelEtAl2020eLife, KorabelEtAl2021EntropyEndosomes, BalcerekEtAl2023SwitchingFBM}, see also \cite{SabriXuKrapfWeiss2020PRL} for experimental motivation. Extensions to Riemann--Liouville type processes \cite{woszczek2025_rlFBMre_chaos} are also natural. Further random-parameter and time-change mechanisms have been studied in, e.g., \cite{WangCherstvyChechkinThapaSenoLiuMetzler2020JPA,WangCherstvyLiuMetzler2020PRE,CherstvySafdariMetzler2021JPhysD,CherstvyWangMetzlerSokolov2021PRE,WangMetzlerCherstvy2022PCCP,LiangWangMetzlerCherstvy2023PRE}.  Let us finally mention \cite{BiaginiHuOksendalZhang2008,Mishura2008fBM,Nourdin2012} as major references for FBM.

\paragraph{Outline.} This paper is structured as follows. Section~\ref{sec:smallbarrier} contains the proof of Theorem~\ref{thm:smallbarrier}. The lower bound in Theorem~\ref{Thm} is established in Section \ref{sec_LB}, while the upper bound is proved in Section \ref{sec_UB}.

%%%%%%%%%%%%%%%%%%%%%%%%%%%%%%%%%%%%%%%%%%%%%%%%%%%%%%%%%%
%%%%%%%%%%%%%%%%%%%%%%%%%%%%%%%%%%%%%%%%%%%%%%%%%%%%%%%%%%
\section{Proof of Theorem~\ref{thm:smallbarrier}} \label{sec:smallbarrier} 
The proof of Theorem~\ref{thm:smallbarrier} is based on the following lemma, which can be seen with the help of a simple application of Slepian's lemma. The statement can be found in \cite[Proposition 6]{SlepianZeug}.

\begin{Lemma}
    \label{lem:SlepianZeug}
    Let $0 < H \leq K < 1$ and let $(B_t^H)_{t\geq 0}$, respectively $(B_t^K)_{t\geq 0}$, be a fractional Brownian motion with Hurst exponent $H$, respectively $K$. Then for every $\varepsilon > 0$ we have
    \begin{align*}
    \p\left(\max_{t\in [0,1]}B_t^{H}\leq \varepsilon\right)\leq \p\left(\max_{t\in [0,1]}B_t^{K}\leq \varepsilon\right).
\end{align*}
\end{Lemma}

\begin{proof}[Proof of Theorem~\ref{thm:smallbarrier}]
Lemma~\ref{lem:SlepianZeug} states that $H\mapsto\p_B\left(\max_{t\in [0,1]}B_t^H\leq \varepsilon\right)$ is non-decreasing. Since $\mathcal{H}\leq H_0$ a.s., monotonicity immediately gives the upper bound 
\begin{align*}
    \E_\mathcal{H}\left[\p_B\left(\max_{t\in[0,1]}B_t^\mathcal{H} < \varepsilon\right)\right] &\leq \E_\mathcal{H}\left[\p_B\left(\max_{t\in[0,1]}B_t^{H_0} < \varepsilon\right)\right] %\notag\\
    = \p_B \left(\max_{t\in[0,1]}B_t^{H_0} < \varepsilon\right) %\notag\\
    = \varepsilon^{\frac{1-H_0}{H_0}+o(1)}, %\label{UB}
\end{align*}
for $\varepsilon\rightarrow 0$, where we used (\ref{eqn:molchaneps}). For the lower bound we restrict to Hurst exponents close to $H_0$ and then use monotonicity again. To do so, let $H_0 > \delta > 0$. Then
\begin{align}
    \E_\mathcal{H}\left[\p_B\left(\max_{t\in[0,1]}B_t^\mathcal{H} < \varepsilon\right)\right] &\geq \E_\mathcal{H}\left[\1_{\{\mathcal{H}\in [H_0-\delta, H_0]\}}\p_B\left(\max_{t\in[0,1]}B_t^\mathcal{H} < \varepsilon\right)\right]\notag\\
    & \geq c_\delta\,\p_B\left(\max_{t\in[0,1]}B_t^{H_0-\delta} < \varepsilon\right)\notag\\
    &= c_\delta\,\varepsilon^{\frac{1-(H_0-\delta)}{H_0-\delta}+o(1)},\label{eqn:fapunkt}
\end{align}
for $\varepsilon \rightarrow 0$, where $c_\delta := \p_\mathcal{H}(\mathcal{H}\in [H_0-\delta, H_0])>0$, where we used again (\ref{eqn:molchaneps}). Taking logarithms of (\ref{eqn:fapunkt}), dividing by $\ln \eps$, and letting first $\eps$ and then $\delta$ tend to zero gives the lower bound in the claim.
\end{proof}

\section{Proof of the Lower Bound in Theorem~\ref{Thm}}\label{sec_LB}

The goal of this section is to prove the lower bound in Theorem \ref{Thm}.
Here, we adapt Molchan’s approach \cite{molchan99}, originally developed to find a lower bound for the persistence probability of a fractional Brownian motion with fixed Hurst exponent. In the random-exponent setting, the main idea is to establish a lower bound uniformly for Hurst exponents close to the essential supremum $H_0$ and then average over $\mathcal H$. \\
The proof has three main ingredients. First, we revisit one of the auxiliary steps in Molchan's argument and make it uniform for \(H\) in intervals bounded away from zero. Second, we restrict the random Hurst exponent to values close to \(H_0\), since these values are expected to give the dominant contribution. Finally, we relate the resulting integral estimate to the persistence event by adapting Molchan's original argument.\\
Throughout this section, for $T\ge 0$ we use the shorthand notation
\begin{align*}
    M_T^H := \max_{t\in[0,T]} B_t^H, \qquad
    A_T^H := \max_{t\in[0,T]} |B_t^H|.
\end{align*}

%%%%%%%%%%%%%%%%%%%%%%%%%%%%%%%%%%%%%%%%%%%%%
%%%%%%%%%%%%%%%%%%%%%%%%%%%%%%%%%%%%%%%%%%%%%%%%%%%%%%%%%%
%%%%%%%%%%%%%%%%%%%%%%%%%%%%%%%%%%%%%%%%%%%%%%%%%%%%%%%%%%
\subsection{Preliminaries}\label{sec_Lemmas_LB}
%%%%%%%%%%%%%%%%%%%%%%%%%%%%%%%%%%%%%%%%%%%%%%%%%%%%%%%%%%

In this section, we collect several auxiliary results needed for the proof of the lower bound in Theorem~\ref{Thm}. Some of them are standard facts from the literature, while the remaining ones are proved here. Several of these lemmas will also be used later in the proof of the upper bound. Throughout this section, whenever fractional Brownian motion is involved, we assume a fixed (non-random) Hurst exponent. 

We begin with a brief summary of basic properties of Gaussian processes that will be used throughout the lower-bound argument. The first one is the so-called Borell-TIS inequality in the version from \cite[Theorem 2.5.8]{GineNickl2021}. 

\begin{Lemma}\label{Borel-TIS}
    Let $(X_t)_{t\in I}$ be a centered Gaussian process on the index set $I \subseteq \R$ with $M:=\sup_{t\in I}|X_t|$ a.s.\ finite and let $\sigma^2_I := \sup_{t\in I}\E\left[X_t^2\right]$. Then $\E\left[\sup_{t\in I}X_t\right] < \infty$ and $\sigma_I < \infty$. Moreover, for each $u \geq 0$,
    \begin{align*}
        \p\left(M -\E[M]> u \right) \leq e^{-\frac{u^2}{2\sigma_I^2}}\quad \text{and} \quad
    \p\left(M -\E[M]< -u \right) \leq e^{-\frac{u^2}{2\sigma_I^2}}.
    \end{align*}
\end{Lemma}

Throughout, we use the following tail-based characterization of subgaussian random variables.
A random variable \(X\) is called subgaussian with parameter \(K>0\) if
\begin{align}
    \p\left(X\geq t\right)\leq e^{-\frac{t^2}{2K^2}}
    \quad \text{and} \quad
    \p\left(X\leq -t \right)\leq e^{-\frac{t^2}{2K^2}},
    \qquad \text{for all } t \geq 0 .
    \label{sub}
\end{align}
For centered random variables, this notion is equivalent to the usual
moment-generating-function characterization in the sense that if one
characterization holds with parameter \(K\), then the other holds with parameter
\(CK\) for some universal constant \(C>0\); see
\cite[Lemmas 1.3 and 1.5]{RigolletHuetter2023}.
In particular, the following direction \cite[Lemma 1.5]{RigolletHuetter2023} will be used repeatedly.

\begin{Lemma}\label{lem:subg-mgf}
Let $X$ be a centered subgaussian random variable with parameter $K$, i.e.\ $X$ satisfies (\ref{sub}) and $\E[X]=0$. Then, for any $\theta>0$,
\[
\E\big[e^{\theta X}\big]\le \exp\!\Big(4\theta^2K^2\Big).
\]
\end{Lemma}

We will further use the following important estimate for the expected maximum of fractional Brownian motion repeatedly in the sequel. It is stated as Theorem 1 in \cite{BMNZ2017}.

\begin{Lemma}\label{Max_bound}
    Let $(B_t^H)_{t\geq 0}$ be a fractional Brownian motion with Hurst exponent $H$. Then
    \begin{align*}
        \frac{1}{2\sqrt{H\pi e\ln(2)}} \leq \E \left[\max_{t \in [0,1]}B^H_t\right]< \frac{16.3}{\sqrt{H}}.
    \end{align*}
\end{Lemma}

We now apply Lemma \ref{lem:subg-mgf} to the maximum of the absolute value of an FBM on $[0,1]$.
Using that this maximum is a subgaussian random variable and combining this with the estimate on its mean from Lemma \ref{Max_bound}, we obtain an explicit bound on its exponential moments. This is stated in the next lemma.

\begin{Lemma} \label{Lem:e_hoch_max}
     Let $(B_t^H)_{t \geq 0}$ be an FBM with Hurst exponent $H$. We have for all $\theta > 0$ that
    \begin{align*}
        \E\left[e^{\theta A_1^H}\right]
        \le 2\,\exp\!\left(\frac{16.3\,\theta}{\sqrt{H}} + 4 \theta^2\right). % \label{Lem:e_hoch_max:1}
    \end{align*}
\end{Lemma}

\begin{proof}
The proof follows from the inequality $e^{\theta A_1^H} \leq e^{\theta M_1^H} + e^{\theta \max (-B^H_t)}$, the fact that $\max_{t\in [0,1]} (-B^H_t)$ has the same distribution as $M_1^H$, the equality $M_1^H=M_1^H-\E M_1^H+\E M_1^H$, and the fact that $X:=M_1^H-\E M_1^H$ is subgaussian with parameter $1$, by Borell-TIS (Lemma~\ref{Borel-TIS}), so that we can apply Lemma~\ref{lem:subg-mgf} (to $X$) and Lemma~\ref{Max_bound} (to $\E M_1^H$).
\end{proof}

\subsection{A Uniform Version of Statement 1 in \texorpdfstring{\cite{molchan99}}{Molchan (1999)}}
Our next goal is to establish a more general variant of Statement~1 in Molchan \cite{molchan99}, which will be a key ingredient in the proof of the lower bound in Theorem~\ref{Thm}. Throughout this section, whenever we discuss fractional Brownian motion, the Hurst exponent is again kept fixed (i.e.\ non-random).

We will use an auxiliary tool, namely a slightly less general version of Lemma~1.1 in \cite{GarsiaRodemichRumsey1971}, commonly referred to as the Garsia-Rodemich-Rumsey inequality (GRR).

\begin{Lemma}\label{lem:GRR}
Let $\Psi$ be a continuous, strictly increasing function on $[0,\infty)$ such that
$\Psi(0)=0$ and $\lim_{t\to\infty}\Psi(t)=\infty$ and let $\phi: [0,1]\rightarrow [0,1]$ be a continuous and non-decreasing function with $\phi(0)=0$.
Let $f\in C([0,1],\mathbb{R}^d)$. If there is a constant $C>0$ such that
\begin{align}
\int_0^1\!\!\int_0^1
\Psi\!\left(\frac{|f(t)-f(s)|}{\phi(|t-s|)}\right)\,\dd s\,\dd t \le C, \label{2.1.1}
\end{align}
then for all $0\le s\le t\le 1$,
\begin{align}
|f(t)-f(s)|
\le 8 \int_0^{|t-s|}
\Psi^{-1}\!\left(\frac{4C}{u^2}\right)\, \dd \phi(u). \label{2.1.2}
\end{align}
\end{Lemma}

The next lemma concerns an auxiliary step in Molchan's argument. More precisely, we study the large-$T$ behavior of
\[
\E\left[\frac{\int_0^1 B_u^H e^{T^H B_u^H}\,\dd u}{\int_0^1 e^{T^H B_u^H}\,\dd u}\right].
\]
For fixed \(H\), this quantity converges to \(\E[M_1^H]\), since the exponential weight increasingly concentrates near the points where \(B^H\) attains its maximum on \([0,1]\). In our setting, however, it is not enough to know this convergence for each fixed \(H\). Since we later average over the random Hurst exponent, we need the error term to vanish uniformly for \(H\in[H_1,H_0]\), where \(H_1>0\) is fixed. The following lemma provides precisely this uniform version. 

\begin{Lemma}\label{Lem:Conv_Speed}
     Let $H \in [H_1,H_0]\cap (0,1)$. Let furthermore $(B_t^H)_{t\geq 0}$ be a fractional Brownian motion. Then there exists a function $g:\R_{\geq 0}\times [H_1,H_0]\cap(0,1)\rightarrow\R$ such that 
    \begin{align*}
        \E\left[\frac{\int_0^1B_u^He^{T^HB_u^H}\dd u}{\int_0^1e^{T^HB_u^H}\dd u}\right]  = \E\left[M_1^H\right] + g(T,H),
    \end{align*}
    with
    \begin{align*}
        &\sup_{H\in [H_1,H_0]\cap (0,1)}\left| g(T,H)\right| \rightarrow 0, \qquad \text{as $T \rightarrow \infty$}.
        % \\&\sup_{H\in [H_1,H_0)}\left| g(T,H)\right| \rightarrow 0 \quad \text{as $T \rightarrow \infty$, for $H_0 = 1$.}
    \end{align*}
\end{Lemma}

\begin{proof}
    Fix $H\in[H_1,H_0]\cap (0,1)$ and fix $\delta>0$ (chosen later depending on $\varepsilon$).

\noindent\textit{Step 1: Splitting the expectation and isolating the contribution near the maximizer.}\\
    We define the set of points where the process is within distance $\delta > 0$ of its maximum on $[0,1]$ as
    \[
        I_\delta^H := \left\{ t \in [0,1] : M_1^H - B_t^H \leq \delta \right\}.
    \]
    We note that:
    \begin{enumerate}
        \item On $I_\delta^H$: $B_t^H \geq M_1^H - \delta$, hence $e^{T^HB_t^H} \geq e^{T^H(M_1^H - \delta)}$.
        \item On $(I_\delta^H)^c$: $B_t^H \leq M_1^H - \delta$, hence $e^{T^HB_t^H} \leq e^{T^H(M_1^H - \delta)}$.
    \end{enumerate}

    The main estimate is as follows:
    \begin{align*}
         \Bigg|\frac{\int_0^1 B_u^He^{T^HB_u^H}\dd u}{\int_0^1e^{T^HB_u^H}\dd u}- M_1^H\Bigg|
         &=  \frac{\int_0^1(M_1^H - B_u^H)e^{T^HB_u^H}\dd u}{\int_0^1e^{T^HB_u^H}\dd u} \\
         &\leq \frac{\int_{I_\delta^H}(M_1^H - B_u^H)e^{T^HB_u^H}\dd u}{\int_{I_\delta^H}e^{T^HB_u^H}\dd u}
             +  \frac{\int_{(I_\delta^H)^c}(M_1^H - B_u^H)e^{T^HB_u^H}\dd u}{\int_{I_{\delta/2}^H}e^{T^HB_u^H}\dd u} \\
         &\leq \delta +  2\,A_1^H\,\frac{e^{-\frac{\delta T^H}{2}}}{\lambda(I_{\delta/2}^H)},
    \end{align*}
    where we used $(a)$ and $(b)$ in the last step and where $\lambda$ denotes the Lebesgue measure. Taking expectations and applying the Cauchy-Schwarz inequality yields
    \begin{align}
         &\E\left[\bigg|\frac{\int_0^1B_u^He^{T^HB_u^H}\dd u}{\int_0^1e^{T^HB_u^H}\dd u}- M_1^H\bigg|\right] % \notag \\
         % &\quad
         \leq\delta
         + \left(\E\big[(2A_1^H)^2\big]\right)^{1/2}
           \left(\E\left[\lambda(I_{\delta/2}^H)^{-2}\right]\right)^{1/2}
           \,e^{-\frac{\delta T^H}{2}}.\label{Lem:Conv_Speed_1}
    \end{align}
    Using $x \leq e^x$ for $x \in \R$, Lemma~\ref{Lem:e_hoch_max} as well as the fact that for any $a, b, c > 0$ there exists some $c_0 > 0$ such that $a e^{\frac{b}{\sqrt{H}}+c}\leq e^{\frac{c_0}{\sqrt{H}}}$, we can estimate the first expectation of the r.h.s.\ of (\ref{Lem:Conv_Speed_1}) as follows
    \begin{align}
         \left(\E\big[(2A_1^H)^2\big]\right)^{1/2} \leq \E\left[e^{4A_1^H}\right]^{1/2} \leq e^{\frac{c_0}{\sqrt{H}}}\leq e^{ \frac{c_0}{\sqrt{H_1}}}\label{Lem:Conv_Speed_2},
    \end{align}
    where $c_0 > 0$ is a constant that does not depend on $H$ or $T$.  Our goal in the remainder of the proof is to establish a suitable upper bound for $\E\left[\lambda(I_{\delta/2}^H)^{-2}\right]$ which will allow us to continue the estimate in \eqref{Lem:Conv_Speed_1}.
    
\noindent\textit{Step 2: A lower bound on the Lebesgue measure of $I_{\delta/2}$.}\\
    For any $\gamma < H$, the Hölder constant
    \[
        C_{\gamma}(H):= \sup_{\substack{s,t\in[0,1]\\ s\neq t}}\frac{|B_t^H-B_s^H|}{|t-s|^\gamma}
    \]
    is $\p$-almost surely finite, see \cite[Section 1.16]{Mishura2008fBM}. Let $\tau$ be a (measurable) point where the maximum $M_1^H$ is attained.  
    If $|t-\tau|\leq r$, then $M_1^H-B_t^H \leq  C_\gamma(H) r^\gamma$. Set
    \[
        r := \left( \frac{\delta}{2\,C_\gamma(H)}\right)^{1/\gamma}.
    \]
    Then $t\in [\tau-r, \tau+r]\cap [0,1]$ implies that 
    \begin{align*}
    M_1^H - B_t^H \leq C_\gamma(H) r^\gamma = \frac{\delta}{2}, \qquad \text{$\p$-a.s.,}
       %  |M_1^H - B_t^H| = |B_\tau^H - B_t^H| \leq C_\gamma(H)|t-\tau|^{\gamma} \leq C_\gamma(H)r^\gamma \leq \delta/2 \quad \text{$\p$-a.s.,}
    \end{align*}
    i.e.\ $t \in I_{\delta/2}^H$ $\p$-a.s. Since we have that
    \begin{align*}
        \lambda([\tau-r,\tau+r]\cap [0,1]) \geq \min(1, r),
    \end{align*}
    we obtain the following bound on the Lebesgue measure of $I_{\delta/2}^H$:
    \begin{align}
        \lambda(I_{\delta/2}^H) &\geq \min\left(1,\left(\frac{\delta}{2\,C_\gamma(H)}\right)^{1/\gamma}\right),\notag
    \end{align}
    so that
    \begin{align}
        \lambda(I_{\delta/2}^H)^{-2} &\leq \max\left(1,\ 2^{2/\gamma} C_\gamma(H)^{2/\gamma} \delta^{-2/\gamma}\right).  \label{Lem:Conv_Speed_3}  
    \end{align}
    
\noindent\textit{Step 3: Controlling the moment of $C_\gamma(H)$.}\\
    Our next goal is to control the expectation of $C_\gamma(H)^{2/\gamma}$. To this end, we apply the Garsia-Rodemich-Rumsey inequality (Lemma \ref{lem:GRR}) pathwise with $f(t):= B_t^H$,
\(\Psi(u)=u^q\) and \(\phi(u)=u^\alpha\) for parameters $\alpha = \alpha(H),q = q(H) > 0$ that depend on $H$ and are going to be chosen later. Using the stationary increments of FBM, we have
    \[
        \E\big[|B_t^H-B_s^H|^q\big] = K_q |t-s|^{qH}, \qquad \text{for all $s, t \in [0,1]$},
    \]
    where $K_q := 2^{q/2}\frac{\Gamma(\frac{q+1}{2})}{\sqrt{\pi}}$.
    With that, we obtain for $H>\alpha$
    \begin{align}
        \E\left[\int_0^1\int_0^1 \frac{|B_t^H - B_s^H|^q}{|t-s|^{\alpha q}}\dd s \dd t \right]&=K_q\int_0^1\int_0^1\frac{|t-s|^{Hq}}{|t-s|^{\alpha q}}\dd s \dd t \leq K_q, 
        \label{Lem:Conv_Speed_4}
    \end{align}
    This implies that the random variable $U_{q,\alpha}(H):= \int_0^1\int_0^1 \frac{|B_t^H-B_s^H|^q}{|t-s|^{\alpha q}}\dd s \dd t$ is a.s.\ finite. Applying the GRR Lemma (Lemma~\ref{lem:GRR}) pathwise yields:
\[
|B_t^H(\omega)-B_s^H(\omega)|
\le 8\!\int_0^{|t-s|}
\Big(\tfrac{4\,U_{q,\alpha}(H)(\omega)}{u^2}\Big)^{\!1/q}\, \dd \!\big(u^\alpha\big).
\]
Since $\dd (u^\alpha)=\alpha u^{\alpha-1}\dd u$, we obtain the control
\[
|B_t^H-B_s^H|
\le C_{q,\alpha}\,U_{q,\alpha}(H)^{1/q}\,|t-s|^{\,\gamma}
\quad\text{a.s.\ for all }s,t\in[0,1],
\]
with exponent $\gamma:=\alpha-\frac{2}{q}>0$ (where $\alpha, q$ are such that $\gamma < H$) and constant
\[
C_{q,\alpha}:=\frac{8\,\alpha\,4^{1/q}}{\alpha-\frac{2}{q}}.
\]
Equivalently,
\begin{align*}
C_\gamma(H):=\sup_{s\neq t}\frac{|B_t^H-B_s^H|}{|t-s|^\gamma}
\;\le\; C_{q,\alpha}\,U_{q,\alpha}(H)^{1/q}
\qquad\text{a.s.}
%\label{Lem0.7:1}
\end{align*}
We set $q = 2/\gamma$, so that
\begin{align}
\mathbb E\big[C_\gamma(H)^{2/\gamma}\big]
\le C_{q,\alpha}^{\,q}\,\mathbb E\big[U_{q,\alpha}(H)^{q/q}\big]
= C_{q,\alpha}^{\,q}\,\mathbb E[U_{q,\alpha}(H)].
\label{Lem0.7:2}
\end{align}
%\medskip

\noindent\textit{Step 4: Choosing the remaining parameters to obtain a uniform moment bound.}
\\ We now choose parameters so that the Hölder exponent $\gamma$ takes a simple form. Set $\alpha = H/2$ and $\gamma = H/4$, which gives $q = 8/H$ (this satisfies the conditions from above: $H-\alpha>0$, $\gamma=\alpha-2/q>0$, $\gamma<H$). 
Substituting these values into (\ref{Lem0.7:2})  and using the estimate (\ref{Lem:Conv_Speed_4}) for $\mathbb E [U_{q,\alpha}(H)]$ yields
\begin{align}
\mathbb E\big[C_\gamma(H)^{2/\gamma}\big]
\;&\le\; C_{q,\alpha}^{\,q}\, K_q
\leq  4(16\sqrt{2})^{\frac{8}{H}} \Gamma\left(\frac{\frac{8}{H}+1}{2}\right) 
\leq  4(16\sqrt{2})^{\frac{8}{H_1}} \Gamma\left(\frac{\frac{8}{H_1}+1}{2}\right)=: c_1, \label{Lem:Conv_Speed_5}
\end{align}
and $c_1$ only depends on $H_1$ but not on $H$.

\textit{Step 5: Conclusion.}\\
By combining the estimates (\ref{Lem:Conv_Speed_1}), (\ref{Lem:Conv_Speed_2}), (\ref{Lem:Conv_Speed_3}) and (\ref{Lem:Conv_Speed_5}) from Steps 1--4, using that $\gamma = \alpha-\frac{2}{q}=H/4$, and applying Minkowski's inequality, we obtain
\begin{align*}
    \E\left[\bigg|\frac{\int_0^1B_u^He^{T^HB_u^H}\dd u}{\int_0^1e^{T^HB_u^H}\dd u}- M_1^H\bigg|\right] % \\
         &\quad\leq \delta
         + e^{\frac{c_0}{\sqrt{H_1}}}
          \left(\E\left[ \max\left(1, 2^{2/\gamma}C_\gamma(H)^{2/\gamma}\delta^{-2/\gamma}\right)\right]\right)^{1/2} e^{-\frac{\delta T^H}{2}}\\*
           &\quad \leq \delta + e^{\frac{c_0}{\sqrt{H_1}}}
          \left( 1+ 2^{1/\gamma}\left(\E\left[C_\gamma(H)^{2/\gamma}\right]\right)^{1/2} \delta^{-1/\gamma}\right) e^{-\frac{\delta T^H}{2}}\\*
           &\quad \leq  \delta
         + e^{\frac{c_0}{\sqrt{H_1}}}
           \left(1+ 2^{4/H}\sqrt{c_1}\delta^{-4/H}\right)e^{-\frac{\delta T^H}{2}}\\*
           &\quad \leq \delta + c_2e^{-\frac{\delta T^{H_1}}{2}} + c_3\delta^{-4/H_1}e^{-\frac{\delta T^{H_1}}{2}},
\end{align*}
where $c_2, c_3>0$ do not depend on $H$, only on $H_1$.
We note that this upper bound is valid for $\delta < 1$ and every $H\in [H_1, H_0]\cap (0,1)$ and it does not depend on $H$ (only on $H_1$).
By choosing $\delta :=  \min(\varepsilon/3,1)$ and $T_0$ such that $c_2e^{-\frac{\delta T^{H_1}}{2}}\leq \varepsilon/3$ and $\delta^{-4/H_1}e^{-\frac{\delta T^{H_1}}{2}}\leq \varepsilon/3$ for every $T \geq T_0$, the last expression becomes $\leq\varepsilon$.
\end{proof}

With Lemma~\ref{Lem:Conv_Speed} at hand, we can now establish a more general variant of Statement~1 in \cite{molchan99}. Namely, the next lemma gives us an asymptotic expansion for
\(
\E\big[\big(\int_0^T e^{B_s^H}\,\dd s\big)^{-1}\big]
\),
with an error term that vanishes uniformly in $H\in[H_1,H_0]\cap (0,1)$.

\begin{Lemma} \label{Lem:Statement1}
     Let $(B_t^H)_{t \ge 0}$ be a fractional Brownian motion with Hurst exponent $H\in[H_1,H_0]\cap(0,1)$. There exists a function $g: \R_{\geq 0}\times [H_1,H_0]\cap (0,1) \rightarrow \R$ such that
    \begin{align*}
         \E\left[\left(\int_0^T e^{B_s^H}\dd s\right)^{-1}\right]  
         = \left(\E\left[M_1^H\right] + g(T,H)\right)HT^{H-1} + T^{-1},
    \end{align*}
    where $\sup_{H\in [H_1,H_0]\cap (0,1)}|g(T,H)| \rightarrow 0$, as $T\rightarrow \infty$. 
\end{Lemma}

\begin{proof}
    We proceed as in the proof of Statement 1 of \cite{molchan99}. Using the stationarity of increments, i.e.\ $\left(B_s^H\right)_{s \geq 0} \stackrel{d}{=} \left(B_{T-s}^H - B_T^H\right)_{s\geq 0}$,
    one has
    \begin{align*}
        I & := \E\left[\left(\int_0^T e^{B_s^H}\dd s\right)^{-1}\right]
          = \E\left[\frac{e^{B_T^H}}{\int_0^T e^{B_s^H}\dd s}\right]
           = \E\left[D\varphi(T,\omega)\right].
    \end{align*}
    Here, we set $\varphi(T, \omega) := \ln\left(\int_0^T e^{B_s^H}\dd s\right)$ and denote $D := \frac{\dd}{\dd T}$.  
    We can bound $D\varphi$ uniformly on any interval $[0,T]$ for $T \geq 1$, since
    \begin{align*}
        \frac{e^{B_T^H}}{\int_0^T e^{B_s^H}\dd s} \leq \frac{e^{\max_{t\in [0,T]}B_t^H}}{Te^{\min_{t\in [0,T]}B_t^H}}\leq e^{2\max_{t\in[0,T]}|B_t^H|} \stackrel{d}{=} e^{2T^H\max_{t\in [0,1]}|B_t^H|},
    \end{align*}
    where we used the self-similarity of the FBM in the last step, which shows that $e^{2\max_{t\in[0,T]}|B_t^H|}$ is an integrable majorant, by Lemma \ref{Lem:e_hoch_max}.
    Hence, the operations $\E$ and $D$ are interchangeable.
    Using the self-similarity of $B^H$, we obtain
    \begin{align*}
        I &= D\E\left[\varphi(T,\omega)\right]
          = D \E\left[\ln\left(\int_0^1 e^{T^H B_u^H}\dd u\right) + \ln(T)\right].
    \end{align*}
    Again, $D$ and $\E$ can be interchanged, similar to the justification above. This gives
    \begin{align*}
        I = \E\left[\frac{\int_0^1 B_u^H e^{T^H B_u^H}\dd u}{\int_0^1 e^{T^H B_u^H}\dd u}\right] \cdot DT^H + T^{-1},
    \end{align*}
which holds for all $H\in(0,1)$ and all $T>0$. Inserting Lemma \ref{Lem:Conv_Speed} yields the claim.
\end{proof}

%%%%%%%%%%%%%%%%%%%%%%%%%%%%%%%%%%%%%%%%%%
%%%%%%%%%%%%%%%%%%%%%%%%%%%%%%%%%%%%%%%%%%%
%%%%%%%%%%%%%%%%%%%%%%%%%%%%%%%%%%%%%%%%%%%
\subsection{Lower Bound}\label{sec_proof_LB}
In this section, we return to the random-exponent setting and prove the lower bound in Theorem~\ref{Thm}. 
The key input is the fixed-parameter asymptotic expansion from Lemma~\ref{Lem:Statement1}, which holds uniformly for $H$ in subsets of $(0,1)$ that are bounded away from $0$. 
We first show that Hurst exponents close to the essential supremum $H_0$ already yield the correct order for the lower bound. 
More precisely, we restrict to the event $\{\mathcal H\in[H_0-\varepsilon,H_0]\}$ and combine Lemma~\ref{Lem:Statement1} with the previous auxiliary estimates to obtain Lemma~\ref{Lem:H-Expectation}. 
Finally, we use Lemma~\ref{Lem:H-Expectation} to conclude the desired lower bound in Theorem~\ref{Thm}.

\begin{Lemma}\label{Lem:H-Expectation}
 Let $(B_t^\mathcal{H})_{t\geq 0}$ be an FBMRE and let $\varepsilon > 0$ be such that $H_0 > \varepsilon$.
Then 
\[
    \mathbb{E}_\mathcal{H}\Bigg[
      \1_{\{\mathcal{H} \geq H_0-\varepsilon\}}\,
      \mathbb{E}_B\Bigg[\Bigg(\int_0^T e^{B_s^\mathcal{H}}\,\dd s\Bigg)^{-1}\Bigg]
    \Bigg]
    \;\ge\;
    T^{H_0-1+o(1)}.
\]
\end{Lemma}

\begin{proof}
    Let $\delta > 0$ such that $\varepsilon >  \delta$. Then
    \begin{align}
        \E_\mathcal{H}&\left[\1_{\{\mathcal{H} \geq H_0-\eps\}}\E_B\left[\left(\int_0^T e^{B_s^\mathcal{H}}\dd s\right)^{-1}\right]\right] \geq \E_\mathcal{H}\left[\1_{\{\mathcal{H} \geq H_0-\delta\}}\E_B\left[\left(\int_0^T e^{B_s^\mathcal{H}}\dd s\right)^{-1}\right]\right]. \label{eqn:ldf01}
    \end{align}
      By Lemma \ref{Max_bound} and Lemma \ref{Lem:Statement1}, we can further estimate the last term by
    \begin{align}
         & \E_\mathcal{H}\left[\1_{\{\mathcal{H} \geq H_0-\delta\}}\left(\left(\E_B[M_1^\mathcal{H}]+g(T, \mathcal{H}) \right)\mathcal{H}T^{\mathcal{H}-1}+T^{-1}\right)\right]  \notag  \\
          &= \E_\mathcal{H}\left[\1_{\{\mathcal{H} \geq H_0-\delta\}}\left(\E_B[M_1^\mathcal{H}]+g(T, \mathcal{H}) \right)\mathcal{H}T^{\mathcal{H}-1}\right] + \p_\mathcal{H}([H_0-\delta, H_0])\, T^{-1} \notag  \\
          &\geq \E_\mathcal{H}\left[\1_{\{\mathcal{H} \geq H_0-\delta\}}\left(\frac{c_0}{\sqrt{H_0}}-\sup_{H\in [H_0-\delta, H_0]}|g(T, H)| \right)\mathcal{H}T^{\mathcal{H}-1}\right], \label{eqn:ldf02}
    \end{align}
    for some $c_0 > 0$ that does not depend on $T$ nor on $H$. Because of Lemma \ref{Lem:Statement1}, we know that for large enough $T$
    \begin{align*}
        \sup_{H\in [H_0-\delta, H_0]}|g(T, H)| \leq \frac{c_0}{2\sqrt{H_0}},
    \end{align*}
    so that we can estimate (combining (\ref{eqn:ldf01}) and (\ref{eqn:ldf02}))
    \begin{align*}
       \E_\mathcal{H}\left[\1_{\{\mathcal{H} \geq H_0-\varepsilon\}}\E_B\left[\left(\int_0^T e^{B_s^\mathcal{H}}\dd s\right)^{-1}\right]\right]
        &\geq \frac{c_0}{2\sqrt{H_0}}\,\E_\mathcal{H}\left[\1_{\{\mathcal{H} \geq H_0-\delta\}}\mathcal{H}T^{\mathcal{H}-1}\right]
        \\
        &\geq \frac{c_0}{2\sqrt{H_0}}\ \P_\mathcal{H}( [H_0-\delta,H_0])\,(H_0-\delta) T^{H_0-\delta-1}.
        \end{align*}
        Taking logarithms, dividing by $\ln T$, and letting first $T\to\infty$ and then $\delta\to 0$ shows the claim.
\end{proof}

%%%%%%%%%%%%%%%%%%%%%%%%%%%%%%%%%%%%%%%%%%%%%%%%%%%%%%%%%%
%%%%%%%%%%%%%%%%%%%%%%%%%%%%%%%%%%%%%%%%%%%%%%%%%%%%%%%%%%
%%%%%%%%%%%%%%%%%%%%%%%%%%%%%%%%%%%%%%%%%%%%%%%%%%%%%%%%%%

With the previous lemmas at hand, we are finally ready to prove the lower bound in Theorem \ref{Thm}.

\begin{proof}[Proof of the lower bound in Theorem~\ref{Thm}] 
We proceed similarly to the proof of Lemma~1 in \cite{molchan99}. Let $H_0 > \varepsilon > 0$. By Lemma~\ref{Lem:H-Expectation},
\begin{align}
T^{H_0-1+o(1)} \leq \E_\mathcal{H}\!\left[\1_{\{\mathcal{H} \geq H_0-\varepsilon\}} \E_B\!\left[\left(\int_0^T e^{B_s^\mathcal{H}} \dd s\right)^{-1}\right]\right] 
=: I_1 + I_2, 
\label{13}
\end{align}
where
\begin{align*}
I_1 &:= \E_\mathcal{H}\!\left[f_\varepsilon(\mathcal{H})\, \E_B\!\left[(\cdot)^{-1} \1_{\{M_T^\mathcal{H} < a_T\}}\right]\right], \qquad
I_2 := \E_\mathcal{H}\!\left[f_\varepsilon(\mathcal{H})\, \E_B\!\left[(\cdot)^{-1} \1_{\{M_T^\mathcal{H} \geq a_T\}}\right]\right],
\end{align*}
 with $a_T$ depending on $T$ and to be defined later, and $f_\varepsilon(\mathcal{H}) := \1_{\{\mathcal{H} \geq H_0-\varepsilon\}}$. 
The role of the following decomposition is to compare the inverse functional with the persistence event. The term \(I_1\) corresponds to the event where the maximum stays below a suitable barrier; on this event, the functional can be controlled directly and yields the desired persistence probability. The term \(I_2\) corresponds to the complementary event where the maximum exceeds this barrier and will be shown to be negligible.
 \\[0.5em]
\textit{Step 1: Estimation of $I_1$.} Define
\[
\xi^\mathcal{H}_T := \left(\int_0^T e^{B_s^\mathcal{H}} \dd s\right)^{-1},
\]
so that
\[
I_1 \leq \E_\mathcal{H}\!\left[f_\varepsilon(\mathcal{H}) \E_B\left[\xi^\mathcal{H}_1 \1_{\{M_T^\mathcal{H} < a_T\}}\right]\right].
\]
We have $\xi^\mathcal{H}_1 \leq \exp(\sqrt{8c_0 \ln(T)})$ on the set
\[
\Omega^\mathcal{H} := \left\{\omega : \min_{t \in [0,1]} B_t^\mathcal{H} > -\sqrt{8c_0\ln(T)} \right\},
\]
where $c_0>0$ is a constant that is to be chosen later.
Therefore,
\begin{align}
I_1 \leq \E_\mathcal{H}\!\left[f_\varepsilon(\mathcal{H}) \p_B(M_T^\mathcal{H} < a_T) e^{\sqrt{8c_0 \ln(T)}}\right] 
+ \E_\mathcal{H}\left[f_\varepsilon(\mathcal{H}) \E_B[\xi^\mathcal{H}_1 \1_{(\Omega^\mathcal{H})^c}]\right]. \label{punkt}
\end{align}
By Cauchy-Schwarz and the symmetry of FBM,
\begin{align*}
\E_\mathcal{H}&[\E_B[f_\varepsilon(\mathcal{H}) \xi^H_1 \1_{(\Omega^\mathcal{H})^c}]]\\
&\leq \left(\E_\mathcal{H}\!\left[f_\varepsilon(\mathcal{H}) \p_B((\Omega^\mathcal{H})^c)\right]\right)^{1/2} 
\cdot \left(\E_\mathcal{H}\!\left[ \E_B[(\xi^\mathcal{H}_1)^2]\right]  \right)^{1/2}\\
&\leq \left(\E_\mathcal{H}\!\left[f_\varepsilon(\mathcal{H}) \p_B\!\left(M_1^\mathcal{H} > \sqrt{8c_0 \ln(T)}\right)\right] \right)^{1/2}
\cdot \left(\E_\mathcal{H}\!\left[\E_B \left[e^{2 A_1^{\mathcal{H}}}\right]\right]\right)^{1/2}.
\end{align*}
We can further estimate using Lemma \ref{Lem:e_hoch_max} to handle $\E_B\left[e^{2 A_1^{\mathcal{H}}}\right]$ and end up at
\begin{align}
    \E_\mathcal{H}&[\E_B[f_\varepsilon(\mathcal{H}) \xi^H_1 \1_{(\Omega^\mathcal{H})^c}]] \leq \left(\E_\mathcal{H}\!\left[f_\varepsilon(\mathcal{H}) \p_B\!\left(M_1^\mathcal{H} > \sqrt{8c_0 \ln(T)}\right)\right] \right)^{1/2}
\cdot  e^{\frac{16.3}{\sqrt{H_0-\varepsilon}}+8}. \label{Lem_LB_1}
\end{align}
To treat the remaining expectation in the inequality above, let $T$ be large enough, such that $\sqrt{2c_0\ln(T)}\geq \frac{16.3}{\sqrt{H_0-\varepsilon}}$,
since then, by Lemma \ref{Max_bound}, we have on the set $\{\mathcal{H}\geq H_0-\varepsilon\}$ that, $\P_\mathcal{H}$-a.s.\
\begin{align*}
    \frac{\sqrt{8c_0\ln(T)}}{2}=\sqrt{2c_0\ln(T)}\geq \frac{16.3}{\sqrt{H_0-\varepsilon}}  \geq \frac{16.3}{\sqrt{\mathcal{H}}}\geq \E_B\left[M_1^\mathcal{H}\right].
\end{align*}
Using this and the Borell-TIS inequality (Lemma \ref{Borel-TIS}), we can estimate the expectation as follows
\begin{align}
    \E_\mathcal{H}\!\left[f_\varepsilon(\mathcal{H}) \p_B\!\left(M_1^\mathcal{H} > \sqrt{8c_0 \ln(T)}\right)\right]  &\leq \E_\mathcal{H}\!\left[f_\varepsilon(\mathcal{H}) \p_B\!\left(M_1^\mathcal{H} > \E_B\left[M_1^\mathcal{H}\right]+\sqrt{2c_0 \ln(T)}\right)\right] \notag\\
    &\leq  T^{-c_0}. \label{Borel_TIS_estimation}
\end{align}
Substituting this in (\ref{Lem_LB_1}) implies
\[
\E_\mathcal{H}[f_\varepsilon(\mathcal{H}) \E[\xi^\mathcal{H}_1 \1_{(\Omega^\mathcal{H})^c}]] \leq c_1T^{-c_0/2},
\]
for some constant $c_1 > 0$.
By inserting this into (\ref{punkt}), we obtain
\begin{align}
I_1 \leq \E_\mathcal{H}\!\left[f_\varepsilon(\mathcal{H}) \p_B(M_T^\mathcal{H} < a_T)\right] e^{ \sqrt{8 c_0 \ln(T)}} + c_1 T^{-c_0/2}.
\label{14}
\end{align}
\textit{Step 2: Estimation of $I_2$.} The estimation of $I_2$ relies on the fact that, when $M_T^\mathcal{H} > a_T$, the path of $B_t^\mathcal{H}$ remains above the level $\rho_T a_T$ (for some $0 < \rho_T < 1$ that we will choose later) over an interval of suitable length. Assume for the sake of notation that $T$ is an integer. The proof for non-integer $T$ is completely analogous, we only use $\lfloor T \rfloor$ instead of $T$ which leads to the same asymptotics in (\ref{16.2}). \\
We divide $[0,T]$ into $T$ equal sub-intervals $\Delta_i$, $i=1,\ldots,T$. Let $R_i^\mathcal{H} := \max_{t\in [i-1, i]}B_t^\mathcal{H}-\min_{t\in [i-1,i]}B_t^\mathcal{H}$ be the range of $B_t^\mathcal{H}$ on $\Delta_i$, set $\overline{\rho}_T := 1 - \rho_T$, and define
\[
D^\mathcal{H} := \bigcup_i \{R_i^\mathcal{H} > \overline{\rho}_T a_T\}.
\]
We now choose $a_T:= \frac{\sqrt{8c_2\ln(T)}}{\overline{\rho}_T}$ for some constant $c_2 > 0$ chosen later, so that
\begin{align}
\overline{\rho}_T a_T =  \sqrt{8c_2 \ln(T)}.
\label{15}
\end{align}
We furthermore define $\nu^\mathcal{H}(s,t) := B_{s}^\mathcal{H} - B_{t}^\mathcal{H}$, $(s,t) \in [0,1]^2$. The stationarity of increments yields
\begin{align}
\E_\mathcal{H}[f_\varepsilon(\mathcal{H}) \p_B(D^\mathcal{H})] 
&\leq \E_\mathcal{H}\!\left[f_\varepsilon(\mathcal{H})\sum_{i=0}^{T-1} \p_B\!\left(\max_{(s,t) \in [i, i+1]^2}(B_{s}^\mathcal{H}-B_{t}^\mathcal{H})   > \overline{\rho}_Ta_T\right)\right] \notag\\
&=\E_\mathcal{H}\!\left[f_\varepsilon(\mathcal{H})\cdot T\cdot  \p_B\!\left(\max_{[0,1]^2} \nu^\mathcal{H} > \sqrt{8c_2\ln(T)}\right)\right].\label{16.1}
\end{align}

To estimate this expression, we proceed very similarly to the estimation in (\ref{Borel_TIS_estimation}). 
For $T$ large enough such that $\sqrt{2c_2\ln(T)}\geq 2\frac{16.3}{\sqrt{H_0-\varepsilon}}$ and on the set $\{\mathcal{H}\geq H_0-\varepsilon\}$ we have, by Lemma \ref{Max_bound}, $\P_{\mathcal{H}}$-a.s.\
\begin{align*}
    \sqrt{2c_2\ln(T)}  &\geq 2 \E_B[M_1^\mathcal{H}] % \\
     = \E_B\left[ \max_{t\in [0,1]}B_t^\mathcal{H}\right]- \E_B\left[\min_{t\in[0,1]}B_t^\mathcal{H}\right] %  \\
    = \E_B\left[\max_{(s,t)\in [0,1]^2}\nu^\mathcal{H}(s,t)\right].
\end{align*}
We can therefore further estimate (\ref{16.1}) with the Borell-TIS inequality (Lemma~\ref{Borel-TIS}):
\begin{align}
    \E_\mathcal{H}[f_\varepsilon(\mathcal{H}) \p_B(D^\mathcal{H})] % \notag\\   
& \leq \E_\mathcal{H}\!\left[f_\varepsilon(\mathcal{H})\cdot T\cdot  \p_B\!\left(\max_{[0,1]^2} \nu^\mathcal{H} > \sqrt{8c_2\ln(T)}\right)\right]\notag\\
&\leq T\E_\mathcal{H}\!\left[f_\varepsilon(\mathcal{H}) \p_B\!\left(\max_{[0,1]^2} \nu^\mathcal{H} > \sqrt{2c_2\ln(T)} + \E_B\left[\max_{[0,1]^2}\nu^\mathcal{H}\right]\right)\right]\notag\\
&\leq T^{-(c_2-1)}, \label{16.2}
\end{align}
where we used that $\sup_{(s, t)\in [0,1]^2}\E_B[\nu^\mathcal{H}(s,t)^2]\leq 1$. On the other hand, the event $(D^\mathcal{H})^c \cap \{M_T^\mathcal{H} > a_T\}$ is a subset of
\[
C^\mathcal{H} := \{\lambda(\{s \in [0,T]: B_s^\mathcal{H} > \rho_T a_T\}) \geq 1\},
\]
since any maximizer $t^\ast\in[0,T]$ of $B^\mathcal{H}$ belongs to one of the intervals $\Delta_i$, $i=1,\ldots,T$, say $\Delta_j$ and on $(D^\mathcal{H})^c\cap \{M_T^\mathcal{H}>a_T\}$ that means that we have $B_t^\mathcal{H} > \rho_T a_T$ for every $t\in \Delta_j$ and $\lambda(\Delta_j) = 1$. Thus,
\begin{align}
I_2 &= \E_\mathcal{H}[f_\varepsilon(\mathcal{H}) \E_B[\xi^\mathcal{H}_T \1_{\{M_T^\mathcal{H} > a_T\}}]]  \notag\\*
&\leq \E_\mathcal{H}[f_\varepsilon(\mathcal{H}) \E_B[\xi^\mathcal{H}_1 \1_{D^\mathcal{H}}]] + \E_\mathcal{H}[f_\varepsilon(\mathcal{H}) \E_B[\xi^\mathcal{H}_T \1_{C^\mathcal{H}}]]  
=: J_1 + J_2. \label{I2split}
\end{align}
{\it Step 3: Estimation of $J_{1}$ and $J_2$.} For $J_1$, we proceed as for $I_1$ and use (\ref{16.2}):
\begin{align}
J_1 &\leq \E_\mathcal{H}[f_\varepsilon(\mathcal{H}) \p_B(D^\mathcal{H})] e^{\sqrt{8 c_0 \ln(T)}} + c_1 T^{-c_0/2} % \notag \\
\leq T^{-(c_2-1)} e^{\sqrt{8 c_0 \ln(T)}} + c_1 T^{-c_0/2}.
\label{17}
\end{align}
To estimate $J_2$, we note that using $\lambda(S)\geq 1$ on $C^\mathcal{H}$, with the notation $S:=\{s\in [0,T]: B_s^\mathcal{H}> \rho_Ta_T\}$, we have
\begin{align*}
    \E_B\left[\xi_T^\mathcal{H}\1_{C^\mathcal{H}}\right] &= \E_B\left[\left(\int_0^T e^{B_s^\mathcal{H}} \dd s\right)^{-1}\1_{C^\mathcal{H}}\right]
    \leq \E_B\left[\left(\int_S e^{B_s^\mathcal{H}} \dd s\right)^{-1}\1_{C^\mathcal{H}}\right]
    \leq e^{-\rho_Ta_T}.
    \end{align*}
Using this, setting $\rho_T := \frac{\sqrt{\ln(T)}}{1+\sqrt{\ln(T)}}$, and using (\ref{15}) and $\rho_T=1-\overline\rho_T$, we obtain
\begin{align}
J_2 &\leq \E_\mathcal{H}\!\left[f_\varepsilon(\mathcal{H}) e^{-\rho_T a_T}\right] 
\leq e^{-\rho_Ta_T} = T^{-\sqrt{8c_2}}. \label{18}
\end{align}
{\it Step 4: Final argument.} Combining (\ref{13}), (\ref{14}), (\ref{I2split}), (\ref{17}), and (\ref{18}) yields
\begin{align*}
T^{H_0-1+o(1)}
&\leq \E_\mathcal{H}\!\left[f_\varepsilon(\mathcal{H}) \p_B(M_T^\mathcal{H} < c_3 \ln(T))\right] e^{c_4 \sqrt{\ln(T)}} + T^{-2},
\end{align*}
for some constants $c_3, c_4 > 0$ and sufficiently large $T$, provided that $c_0$ and $c_2$ are chosen large enough so that $c_0/2 > 2, \sqrt{8c_2}> 2$ and $c_2-1> 2$. Thus, for large $T$,
\begin{align}
 T^{H_0-1+o(1)}\leq \E_\mathcal{H}\!\left[f_\varepsilon(\mathcal{H}) \p_B(M_T^\mathcal{H} < c_3 \ln(T))\right]. \label{19}
\end{align}
It follows from the self-similarity of FBM that 
\begin{align}
    \E_\mathcal{H}[f_\varepsilon(\mathcal{H})\p_B(M_T^\mathcal{H} < c_3\ln(T))] &= \E_\mathcal{H}[f_\varepsilon(\mathcal{H})\p_B(M_{T'}^\mathcal{H} < 1)], \label{Lem_UB_End}
\end{align}
 where $T' := \frac{T}{(c_3\ln(T))^{1/\mathcal{H}}}$.
Let us now define $T'_\varepsilon:= \frac{T}{(c_3\ln(T))^{1/(H_0-\varepsilon)}}$ and note, that on the event \(\{H\ge H_0-\varepsilon\}\) and for $T \geq 1$, the denominator of $T'$, i.e. \((c_3\ln T)^{1/H}\), is bounded above by \((c_3\ln T)^{1/(H_0-\varepsilon)}\). Hence \(T'\ge T'_\varepsilon\). Since the persistence event over the longer interval \([0,T']\) is contained in the persistence event over the shorter interval \([0,T'_\varepsilon]\), we have
\begin{align}
    \E_\mathcal{H}[f_\varepsilon(\mathcal{H})\p_B(M_{T'}^\mathcal{H} < 1)] 
    &\leq \E_\mathcal{H}[f_\varepsilon(\mathcal{H})\p_B(M_{T'_{\varepsilon}}^\mathcal{H} < 1)]. \label{Lem_UB_End2}
\end{align}
We can substitute (\ref{Lem_UB_End}) and (\ref{Lem_UB_End2}) in (\ref{19}) and use the estimate $f_\varepsilon(H) \leq 1$ to arrive at
\begin{align*}
    T^{H_0-1+o(1)} &\leq \E_\mathcal{H}[f_\varepsilon(\mathcal{H})\p_B(M_{T'_{\varepsilon}}^\mathcal{H} < 1)] \leq \E_\mathcal{H}\left[ \p_B(M_{T'_{\varepsilon}}^\mathcal{H} < 1)\right].
\end{align*}
By substituting $T'_\varepsilon$ in the inequality above, we obtain that for sufficiently large $T'_\varepsilon$
\begin{align*}
   (T'_\varepsilon)^{H_0-1+o(1)}\leq \E_\mathcal{H}[\p_B(M_{T'_{\varepsilon}}^\mathcal{H} < 1)],
\end{align*}
where we used that $T=(T_\varepsilon')^{1+o(1)}$ to rewrite the left hand side of (\ref{19}).
\end{proof}

%%%%%%%%%%%%%%%%%%%%%%%%%%%%%%%%%%%%%%%%%%%%%%%%%%%%%%%%%%
%%%%%%%%%%%%%%%%%%%%%%%%%%%%%%%%%%%%%%%%%%%%%%%%%%%%%%%%%%
%%%%%%%%%%%%%%%%%%%%%%%%%%%%%%%%%%%%%%%%%%%%%%%%%%%%%%%%%%

\section{Proof of the Upper Bound in Theorem~\ref{Thm}}\label{sec_UB}
The goal of this section is to prove the upper bound in Theorem~\ref{Thm}.
The key idea is to reduce the continuous-time event to a persistence event for a discretely sampled path. To treat this discrete-time problem, we refine the approach of \cite{Aurzada18}, developed for persistence probabilities of discretely sampled processes. %The sampling scheme depends on the Hurst exponent (in particular for small values of $H$) and is chosen so that the resulting bounds remain stable uniformly in $H\in (0,H_0)$.
The main additional difficulty compared with the fixed-\(H\) case is that estimates which are harmless for a fixed Hurst exponent may deteriorate as \(H\downarrow0\). We therefore choose the discretization scale in a way that depends on \(H\), and we separate the argument into regimes where different lower bounds for the discrete persistence probability are effective. 

\subsection{Lemmas for the Upper Bound}

As for the lower bound, we begin by collecting several auxiliary results that will be needed in the proof of the upper bound. Throughout this subsection, whenever fractional Brownian motion is involved, the Hurst exponent is kept fixed (i.e.\ non-random). We start by recalling the following estimate, which appears as Theorem~2 in \cite{Aurzada18}.

\begin{Lemma}\label{Thm2}
    Let $(Z_j)_{j\geq 0}$ be a centered process with stationary increments, i.e.\
$(Z_{j+\ell}-Z_\ell)_{j\geq 0} \deq (Z_j)_{j\geq 0}$ for all $\ell\geq 0$. Then for all $a> 0$ and all $n\in \N$:
    \begin{align*}
        \p\left(\max_{j=1,\ldots,n}Z_j \leq -a\right)\leq \frac{\E[\max_{j=1,\ldots,n+1}Z_j]}{an}.
    \end{align*}
\end{Lemma}

As a consequence, we obtain the following corollary.

\begin{Corollary}\label{Cor:No_Cond_on_H}
    Let $(B_t^H)_{t\geq 0}$ be a fractional Brownian motion with Hurst exponent $H\in (0,1)$ and let $n, m \in \N$. Then 
    \begin{align*}
        \p\left(\max_{k=1,\ldots,nm}B_{k/m}^H \leq -1 \right) \leq 2n^{H-1}\frac{\E\left[ \max_{t\in [0,1]}B_t^H\right]}{m}.
    \end{align*}
\end{Corollary}
\begin{proof}
    The process $(Z_j)_{j\geq 0}:=(B_{j/m}^H)_{j\geq 0}$ is a centered process with stationary increments. Therefore, a direct consequence of the previous lemma is
    \begin{align*}
        \p\left(\max_{k=1,\ldots,nm}B_{k/m}^H \leq -1 \right)
        &\leq \frac{\E[\max_{k=1,\ldots,nm+1}B_{k/m}^H]}{nm} \\*
        &\leq \frac{(n+1/m)^H}{n}\frac{\E\left[ \max_{t\in [0,1]}B_t^H\right]}{m}\\*
        &\leq 2n^{H-1}\frac{\E\left[ \max_{t\in [0,1]}B_t^H\right]}{m},
    \end{align*}
    where we used \[ \max_{k=1,\ldots,nm+1}B_{k/m}^H \leq \max_{t\in [0, n+1/m]}B_t^H \stackrel{d}{=} (n+1/m)^H\max_{t\in [0, 1]}B_t^H,\] using $\{1/m, \ldots, n, n+1/m\}\subseteq[0,n+1/m]$ and the self-similarity of FBM. 
\end{proof}

The following lemma, which is Theorem 2 in \cite{BMNZ2017}, will be particularly useful, as it provides a quantitative bound on the discretization error, i.e.\ the expected difference between the continuous-time maximum of an FBM on $[0,1]$ and the maximum over an equidistant grid on this interval.

\begin{Lemma}\label{Lem:Discretisation_Error}
    Let $(B_t^H)_{t\geq 0}$ be a fractional Brownian motion with Hurst exponent $H\in (0,1)$. Then, for any $n \geq 2^{1/H}$, we have
    \begin{align*}
        \E\left[ \max_{t \in [0,1]}B_t^H\right] -\E\left[ \max_{k=1,\ldots,n}B_{k/n}^H\right]\leq 12\,\frac{\sqrt{\ln(n)}}{n^H}.
    \end{align*}
\end{Lemma}

We will also use the following standard estimates for the tail of the standard normal distribution, sometimes referred to as Mills' ratio inequalities, see e.g. $(10)$ in \cite{gordon1941_mills_ratio}.%{Vershynin2025HDP2Draft}. 
 \ If $\Phi$ is the distribution function of a standard normal distribution and $\phi$ the density we have 
    \begin{align}\label{eqn:Mills_Ratio_Inequalities}
        \frac{1}{x+1/x}\,\phi(x) \leq \Phi(-x)\leq \frac{1}{x}\,\phi(x),\qquad \text{for all $x\geq 0$.}
    \end{align}

Using Corollary~\ref{Cor:No_Cond_on_H}, we can bound the probability of remaining below $-1$. To pass from the boundary $-1$ to $1$ is non-trivial and requires a change of measure argument. For this, following the approach of \cite{Aurzada18}, we will need a suitable {\it lower} bound for the persistence probability of the discretely sampled FBM. The lower bound from Section~\ref{sec_LB} is not sufficient here, since we require estimates that also cover very small Hurst exponents. 
The next two lemmas provide the lower bounds on discrete persistence probabilities that are needed in the change-of-measure step below. We split the argument into two regimes. Lemma~\ref{Lem:Case1} treats the case where $H$ satisfies a certain lower bound depending on $n$, while Lemma~\ref{Lem:Case2} covers the complementary regime of very small $H$ relative to $n$

\begin{Lemma}\label{Lem:Case1}
  Fix $\delta>1/4$. Then there exist constants $c>0$ and $n_0\in\mathbb{N}$ such that for all 
  $n\geq n_0$ and all $H$ satisfying
  \begin{align} \label{eqn:rewriteHcondition}
      H \;\geq\; \frac{\delta \ln\!\,\bigl(\ln n\bigr)}{\ln n}
  \end{align}
  we have
  \[
      \mathbb{P}\!\left( \max_{k=1,\dots,n} B_k^H \leq 0 \right)
      \;\geq\;
      \frac{c}{n^3},
  \]
  where $(B_t^H)_{t\in \R}$ is a fractional Brownian motion with Hurst exponent $H\in (0,1)$.
\end{Lemma}

\begin{proof}
    The proof is a refinement of the method of proof of Theorem 4(ii) in \cite{Aurzada18}. In the following, we fix the notation
    \begin{align*}
        S^H_{n} &:= \max_{k=1,\ldots,n}B_k^H, \qquad \text{and}\qquad 
        S^H_{m, n} := \max_{k=m,\ldots,n}B_k^H.
    \end{align*}
    \medskip
    \\ 
    \noindent\textit{Step 1: From $\p(S_n^H < 0)$ to the expected number of right-to-left records.}\\
    Let $T_n^H := \# \{m\in \{1,\ldots,n^2\}\ : \ S^H_{1+m,n+n^2} < B_m^H\}$. Observe that $T_n^H$ represents the (random) number of `record times' $m$, where the maximum $S_{m,n+n^2}$ over the `remaining times' $\{m,\ldots,n+n^2\}$ is attained at the first time index $m$. Note that
    \begin{align}
        n^2 \p\left(S^H_{n}< 0\right) &\geq \sum_{m=1}^{n^2} \p\left( S^H_{n+n^2-m} < 0\right)% \notag\\
        = \sum_{m=1}^{n^2} \p\left(S^H_{m+1,n+n^2}< B_m^H\right)% \notag\\
        = \E[T_n^H], \label{A:11}
    \end{align}
    where we used the stationary increments in the second step. 
    \medskip
    \\ 
\noindent\textit{Step 2: Controlling the drop in the maximum.}\\
    We set $R^H_{T_n^H+1} := \min\{m\in \{n^2+1,\ldots,n+n^2\}\ : \ S^H_{1+m,n+n^2} < B_m^H \}$, with the convention $S_{n+n^2+1, n+n^2}^H = -\infty$ (so that $R  ^H_{T_n^H+1}$ is well-defined and less than or equal to $n+n^2$). That means, $R^H_{T_n^H+1}$ is the first record of $S^H_{1+m,n+n^2} < B_m^H$ after time $n^2$ and we can thus recursively define the $T_n^H$ many records up to time $n^2$ by
    \begin{align*}
        R^H_i &:= \max\{m\in \{1,\ldots,R^H_{i+1}-1\}\ :\ S^H_{1+m,n+n^2} < B_m^H \}\\
        &= \max\{m\in \{1,\ldots,R^H_{i+1}-1\}\ :\ B_{R_{i+1}}^H < B_m^H \},\qquad i=1,\ldots,T_n^H.
    \end{align*} 
    Note that $R_1^H<R_2^H<\ldots<R_{T_n^H+1}^H $ are the successive record times when scanning from right to left, i.e.\ the times $m$ such that $S_{m+1,n+n^2}^H < B_m^H $. Figure~\ref{fig:right-to-left-records} provides a schematic illustration of the right-to-left record times $R_1^H,\ldots,R_{T_n^H+1}^H$ and is included solely to aid intuition.
\begin{figure}[t]
    \centering
    \includegraphics[width=0.7\textwidth]{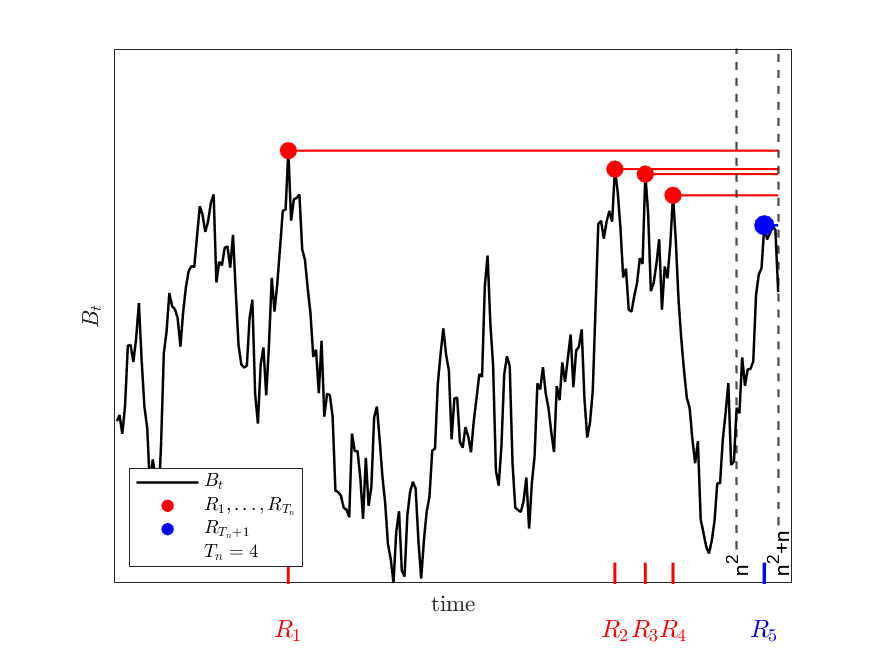}
    \caption{Sketch of the right-to-left record times $R_1^H<\cdots<R_{T_n^H+1}^H$, with reference points at $n^2$ and $n^2+n$ for $H = 1/2$, $n = 15$.}
    \label{fig:right-to-left-records}
\end{figure}
    Note that $B_{R^H_{1}}^H = S^H_{n+n^2}$ and $B_{R^H_{T_n^H+1}}^H = S^H_{1+n^2,n+n^2}$. Hence
    \begin{align}
        S^H_{n+n^2}-S^H_{1+n^2,n+n^2}&= B_{R_1^H}^H - B_{R_{T_n^H +1}^H}^H \notag\\*
        &= \sum_{i=1}^{T_n^H}(B_{R_i^H}^H - B_{R_{i+1}^H}^H)\notag\\*
        &\leq \sum_{i=1}^{T_n^H}(B_{R_i^H}^H - B_{R_{i}^H+1}^H) \notag\\*
        &\leq \max_{m=1,\ldots,n^2}(B_m^H-B_{m+1}^H)\cdot T_n^H,\label{A:12}
    \end{align}
    where we used the fact that $B_{R_{i+1}^H}^H \geq B_{R_i^H +1}^H$, which comes from the definition of $R_i^H$. Observe that (\ref{A:12}) becomes $0 \leq 0$ if $T_n^H = 0$. 
      \medskip
    \\ 
    \noindent
    \textit{Step 3: Isolating the main contribution and a first control of the error term.}\\
    Combining (\ref{A:11}) and (\ref{A:12}), we obtain that for every $b_n > 0$, which will be chosen later,
    \begin{align}
         n^2 &\p\left(S^H_n < 0\right) \notag
         \\
         &\geq \E[T_n^H]\notag\\
         &\geq \E\left[T_n^H \1_{\{\max_{m=1,\ldots,n^2}(B_m^H-B_{m+1}^H) \leq b_n\}}\right] \notag\\
         &\geq b_n^{-1}\E\left[(S^H_{n+n^2} - S^H_{1+n^2,n+n^2}) \1_{\{\max_{m=1,\ldots,n^2}(B_m^H-B_{m+1}^H) \leq b_n\}}\right] \notag 
         \\
          &= b_n^{-1}\E\left[(S^H_{n+n^2} -B_{n+n^2}^H - (S^H_{1+n^2,n+n^2}-B_{n+n^2}^H)) \1_{\{\max_{m=1,\ldots,n^2}(B_m^H-B_{m+1}^H) \leq b_n\}}\right]. \label{A:13}
    \end{align}
    Since $S^H_{1+n^2,n+n^2}-B_{n+n^2}^H \geq 0$, we can estimate
    \begin{align*}
        \E\left[(S^H_{1+n^2,n+n^2} - B_{n+n^2}^H) \1_{\{\max_{m=1,\ldots,n^2}(B_m^H-B_{m+1}^H) \leq b_n\}}\right] &\leq \E\left[S_{1+n^2, n+n^2}^H-B_{n+n^2}^H\right] \\
        &= \E\left[S_{1+n^2, n+n^2}^H-B_{n^2}^H\right] \\
        &= \E[S^H_n],
    \end{align*}
    where we used the fact that $\E[B^H_{n+n^2}] = \E\left[B_{n^2}^H\right] = 0$ and the stationary increments. By plugging this into (\ref{A:13}) and using again the fact that $\E\left[B_{n+n^2}^H\right] = 0$, we obtain
    \begin{align}
        &\quad n^2 \p\left(S^H_n < 0\right) \notag\\
        &\geq b_n^{-1}\left(\E\left[(S^H_{n+n^2} -B_{n+n^2}^H ) \1_{\{\max_{m=1,\ldots, n^2}(B_m^H-B_{m+1}^H) \leq b_n\}}\right]- \E\left[S^H_n\right]\right) \notag\\
        &\geq b_n^{-1}\left(\E\left[S^H_{n+n^2}-S^H_n\right] - \E\left[(S^H_{n+n^2} -B_{n+n^2}^H) \1_{\{\max_{m=1,\ldots, n^2}(B_m^H-B_{m+1}^H) > b_n\}}\right]\right) \notag\\
        &\geq b_n^{-1}\left(\E\left[S^H_{n+n^2}-S^H_n\right]  - \left\|S^H_{n+n^2}-B_{n+n^2}^H\right\|_2 \p\left(\max_{m=1,\ldots,n^2}(B_m^H-B_{m+1}^H) > b_n\right)^{1/2}
        \right), \label{plus}
    \end{align}
    where we used Cauchy-Schwarz in the last step and denote $\|X\|_2 := \E[X^2]^{1/2}$. We further estimate the occurring probability by the subadditivity of $\p$ and the stationary increments:
    \begin{align*}
        \p\left(\max_{m=1,\ldots,n^2}(B_m^H-B_{m+1}^H) > b_n\right) \leq \sum_{m=1}^{n^2} \p\left( -B_1 > b_n\right) = n^2~\p(\mathcal{N}(0,1)> b_n) .  
    \end{align*}
    Continuing in (\ref{plus}), we obtain
    \begin{align}
        &n^2 \p\left(S^H_n < 0\right) %  \notag\\
        \geq b_n^{-1}\left(\E\left[S^H_{n+n^2}-S^H_n\right]  - \left\|S^H_{n+n^2}-B_{n+n^2}^H\right\|_2 \left(n^2~\p(\mathcal{N}(0,1)> b_n)\right)^{1/2} \label{Rest}
        \right).
    \end{align}
    It remains to bound the first term in parentheses from below, which governs the asymptotics, and to control the second term from above, which will be shown to be negligible.
     \medskip
    \\ 
    \noindent
    \textit{Step 4: Control of the leading order term.}\\
    Let us now treat the first expression in the parentheses of (\ref{Rest}). We proceed by inserting the continuous maximum as follows
    \begin{align*}
        \E\left[S^H_{n+n^2}-S^H_n\right] &\geq (n+n^2)^H \E\left[\max_{k=1,\ldots,n+n^2}B_{k/(n+n^2)}^H \right]- n^H \E\left[\max_{t\in [0,1]}B_{t}^H \right]\\
        &= ((n+n^2)^H-n^H)\E\left[\max_{t\in [0,1]}B_t^H \right] \\
        &\quad \quad - (n+n^2)^H\E\bigg[ \max_{t\in [0,1]}B_t^H  
         -\max_{k=1,\ldots,n+n^2}B_{k/(n+n^2)}^H\bigg].
    \end{align*}
    We note that the second expectation is exactly the discretization error, which is, by Lemma \ref{Lem:Discretisation_Error}, bounded by $c_0(n+n^2)^{-H}\sqrt{\ln(n+n^2)}$ for some $c_0 > 0$, independent of $H$, for $(n+n^2)^H \geq 2$. %(i.e.\ at least for $n\geq e^{2^{1/\delta}}$ under (\ref{eqn:rewriteHcondition})) 
    Under condition (\ref{eqn:rewriteHcondition}), which is equivalent to $n^H \geq \ln(n)^\delta$, this is the case for every $n\geq e^{2^{1/\delta}}$, so that we can estimate
    \begin{align*}
        (n+n^2)^H\E\left[ \max_{t\in [0,1]}B_t^H  -\max_{k=1,\ldots,n+n^2}B_{k/(n+n^2)}^H\right] \leq c_0\sqrt{\ln(n+n^2)}.
    \end{align*}
    Using Lemma \ref{Max_bound} and $H<1$, we proceed as follows:
    \begin{align}
         \E\left[S^H_{n+n^2}-S^H_n\right] 
        &\geq ((n+n^2)^H-n^H)\E\left[\max_{t\in [0,1]}B_t^H \right] - c_0 \sqrt{\ln(n+n^2)} \notag\\
        &\geq \frac{1}{c_1\sqrt{H}}\left( (n+n^2)^H - n^H - c_1\sqrt{H}c_0\sqrt{\ln(n+n^2)}\right)\notag\\
        &\geq \frac{1}{c_1\sqrt{H}}\left( n^H(n^H-1) - c_1c_0\sqrt{\ln(n+n^2)}\right),\label{plusplus}
    \end{align}
    for some $c_1 > 0$ independent of $H$.
    Using again (\ref{eqn:rewriteHcondition}), we have for all $n\geq e^{2^{1/\delta}}$ that
    \begin{align*}
        n^H(n^H-1) \geq \ln(n)^\delta(\ln(n)^\delta -1) \geq \frac{1}{2}\ln(n)^{2\delta}.
    \end{align*}
    Combining this with (\ref{plusplus}) gives us
    \begin{align}
        \E\left[S^H_{n+n^2}-S^H_n\right] 
        &\geq \frac{1}{c_1\sqrt{H}}\left( \frac{1}{2}\ln(n)^{2\delta}-c_1c_0 \sqrt{\ln(2n^2)}\right) 
        \geq \frac{c_2}{\sqrt{H}}, \label{main_impact}
    \end{align}
    for some $c_2 > 0$ that does not depend on $H$ and for $n$ large enough, because of $2\delta > \frac{1}{2}$.
      \medskip
    \\ 
    \noindent
    \textit{Step 5: Control of the error term.}\\
    We are left to show that the second part in the parentheses of (\ref{Rest}) is small enough such that we obtain a positive lower bound. First of all, we note that by stationary increments
    \begin{align*}
        \left\|S^H_{n+n^2}-B_{n+n^2}^H\right\|_2 &= (n+n^2)^H \left\|\max_{k=1,\ldots,n+n^2}B_{k/(n+n^2)}^H-B_1^H\right\|_2\\
         &\leq(n+n^2)^H \left\|\max_{t\in[0,1]}(B_t^H-B_1^H)\right\|_2,
    \end{align*}
    where we also used that $\left(\max_{k=1,\ldots,n+n^2}B_{k/(n+n^2)}^H-B_1^H \right)^2\leq \left( \max_{t\in[0,1]}B_t^H-B_1^H\right)^2$ a.s.\ in the last step. We note that, because of the time-reversibility of FBM, $ \max_{t\in[0,1]}(B_t^H-B_1^H)$ has the same distribution as $M_1^H = \max_{s\in [0,1]}B_s^H $, so that we can continue with the latter quantity.
    To proceed, we insert a zero and use Minkowski's inequality to estimate as follows
    \begin{align}
        \left\|M_1^H\right\|_2 =\left\|M_1^H-\E\left[M_1^H\right] + \E\left[M_1^H\right]\right\|_2 &\leq \left\|M_1^H-\E\left[M_1^H\right]\right\|_2+ \E\left[M_1^H\right] \notag \\
        &\leq c_3 + \frac{c_4}{\sqrt{H}}\leq \frac{c_5}{\sqrt{H}}, \label{step5:1}
    \end{align}
    for some $c_3, c_4, c_5 > 0$ that both do not depend on $H$, where we used Lemma \ref{Max_bound} to estimate $\E[M_1^H]$ and the Borell-TIS inequality, i.e. Lemma \ref{Borel-TIS}, to estimate $\left\|M_1^H-\E\left[M_1^H\right]\right\|_2$.
    We are left to bound the part $(n+n^2)^H(n^2~\p(-B_1^H > b_n))^{1/2}$. To do so, we set $b_n := n$ and estimate the probability using (\ref{eqn:Mills_Ratio_Inequalities}) as follows
    \begin{align}
        (n+n^2)^H\left(n^2~\p(\mathcal{N}(0,1) > b_n)\right)^{1/2} &
        %\leq 2n^{3} \left(\p(\mathcal{N}(0,1) \leq -b_n)\right)^{1/2}
        \leq 2n^{3} \sqrt{\frac{c_6}{n}}e^{-\frac{n^2}{4}}, \label{plus3}
    \end{align}
    for some constant $c_6 > 0$ that does not depend on $H$,
    where we used that $(n+n^2)^H \leq 2n^2$ since $H< 1$. We note that the last term in (\ref{plus3}) tends to zero as $n \rightarrow \infty$.
    We can therefore choose a constant $c_7 > 0$ independent of $H$ such that $c_5\cdot  c_7< c_2$, for which for large enough $n$ we have
    \begin{align*}
        (n+n^2)^H\left(n^2~\p(\mathcal{N}(0,1) > b_n) \right)^{1/2}\leq 2n^3 \sqrt{\frac{c_5}{n}}e^{-\frac{n^2}{4}} \leq c_7.
    \end{align*}
    This, together with (\ref{step5:1}) gives us the following upper bound for the error term.
    \begin{align}
        \left\|S^H_{n+n^2}-B_{n+n^2}^H\right\|_2 \left(n^2~\p(\mathcal{N}(0,1)> b_n)\right)^{1/2} \leq \frac{c_5 c_7}{\sqrt{H}}. \label{error_bound}
    \end{align}
     \medskip
    \\ 
    \noindent
    \textit{Step 6: Piecing everything together.}\\
     Inserting (\ref{main_impact}), $b_n = n$ and (\ref{error_bound}) into (\ref{Rest}), gives us for large enough $n$
    \begin{align*}
        \p(S_n^H < 0) \geq \frac{1}{n^3}\frac{1}{\sqrt{H}} \left( c_2-c_5 c_7\right)\geq \frac{c_{8}}{n^3\sqrt{H}}\geq \frac{c_{8}}{n^3},
    \end{align*}
    for some $c_8 > 0$ independent of $H$. 
\end{proof}

We now turn to the complementary small-\(H\) regime, where the Hurst exponent \(H\) satisfies an upper bound depending on \(n\), in contrast to Lemma~\ref{Lem:Case1}. The following lemma provides a lower bound for the discrete persistence probability specifically in this small-\(H\) regime.

\begin{Lemma}\label{Lem:Case2}
 Fix $\delta > 0$. There exist constants $a, b>0$ and $n_0\in\mathbb{N}$ such that for all 
  $n\geq n_0$ and all $H$ satisfying
  \begin{align*}
       H \leq \frac{\delta \ln\!\,\bigl(\ln n\bigr)}{\ln n}
  \end{align*}
  we have
  \[
      \mathbb{P}\!\left( \max_{k=1,\dots,n} B_k^H \leq 0 \right)
      \;\geq\;
      ae^{-b\ln(n)^{1+2\delta}},
  \]
  where $(B_t^H)_{t\geq 0}$ is a fractional Brownian motion with Hurst exponent $H\in (0,1)$.
\end{Lemma}

\begin{proof}
    {\it Step 1: Reduction to a different process with a simpler structure.}
    
     Let $(W_t)_{t\geq 0}$ be a standard Brownian motion and let $(N_i)_{i\in \N}$ be a family of i.i.d.\ standard normal random variables that are independent of $(W_t)_{t\geq 0}$. We follow the approach of Borovkov and Zhitlukhin to show the following estimate derived in the proof of their theorem, namely inequality (3.1) in \cite{BZ18}, which states that the process
    \begin{align*}
        X_k^{H} := \frac{\left(\frac{k}{n}\right)^{H}N_k-W_{(k/n)^{2H}}}{\sqrt{2}},\qquad k=1,\ldots,n,
    \end{align*}
    satisfies $\p\big(\max_{k=1,\ldots,n}B_{k/n}^H \leq x\big)\geq\p\left(\max_{k=1,\ldots,n}X_k^H\leq x\right) $ for any $x\in \R$. This is a consequence of comparing the covariances of the processes and applying Slepian's lemma. Indeed, we have $\V[B_{k/n}^H] = (k/n)^{2H}= \V[X_k^H]$ and for $1 \leq k <\ell \leq n$
    \begin{align*}
        \Cov(B_{k/n}^H, B_{\ell/n}^H) &= \frac{1}{2}\left(\left( \frac{k}{n}\right)^{2H}+\left( \frac{\ell}{n}\right)^{2H}- \left( \frac{\ell-k}{n}\right)^{2H} \right) \\
        &\geq \frac{1}{2}\min\left(\left( \frac{k}{n}\right)^{2H}, \left( \frac{\ell}{n}\right)^{2H}\right) \\*
        &= \Cov(X_k^H, X_\ell^H).
    \end{align*}
    That means that we can estimate (using also the self-similarity):
    \begin{align*}
       \p\left(\max_{k=1,\ldots,n}B_{k}^H \leq 0\right)=\p\left(\max_{k=1,\ldots,n}B_{k/n}^H \leq 0\right) 
        \geq \p\left(\max_{k=1,\ldots,n}X_k^H \leq 0\right).
    \end{align*}
    I.e.\ we are left to find a lower bound for $\p\left(\max_{k=1,\ldots,n}X_k^H \leq 0\right)$.
    
    {\it Step 2: Separation of the ingredients of $X$.}
    
    We will use a standard result from extreme value theory that can be found e.g.\ as Theorem~4.4.8 in \cite{Embrechts1997}, which states that
    \begin{align} \label{eqn:extremvaluetheory}
        \p\left(\max_{k=1,\ldots,n}N_k \leq b_n + x a_n\right)\rightarrow e^{-e^{-x}}, \quad \text{for all $x \in \R$},
    \end{align}
    where $a_n := \frac{1}{\sqrt{2\ln(n)}}$ and $b_n := \sqrt{2\ln(n)}-\frac{\ln(\ln(n))+\ln(4\pi)}{2\sqrt{2\ln(n)}}$. % This allows to  control the first part of the random variables $(X_k)_{k\in \N}$:
    This implies:
    \begin{align*}
        \p\left(\max_{k=1,\ldots,n}B_{k/n}^H \leq 0\right) &\geq \p\left( \max_{k=1,\ldots,n} \frac{\left(k/n\right)^{H}N_k-W_{(k/n)^{2H}}}{\sqrt{2}}\leq 0\right) \\
         &\geq \p\left(\max_{k=1,\ldots,n}(k/n)^{H}N_k \leq b_n, -\min_{k=1,\ldots,n}W_{(k/n)^{2H}}\leq -b_n\right) \\
        &= \p\left(\max_{k=1,\ldots,n}(k/n)^{H}N_k \leq b_n\right)\cdot \p\left( -\min_{k=1,\ldots,n}W_{(k/n)^{2H}}\leq -b_n\right) \\
        &\geq \p\left(\max_{k=1,\ldots,n}N_k \leq b_n+a_n\cdot 0\right)\cdot \p\left( \max_{k=1,\ldots,n}W_{(k/n)^{2H}}\leq -b_n\right),
    \end{align*}
    where we used the independence of $(N_k)_{k\in \N}$ and $(W_t)_{t\geq 0}$  in the second last step. By (\ref{eqn:extremvaluetheory}), we can bound the first expression by a constant $c_{0} > 0$ for large enough $n$. 
    
    {\it Step 3: Computation of the main term.}
    
    We can continue to estimate
    \begin{align*}
            \p\left(\max_{k=1,\ldots,n}B_k^H \leq 0\right) &\geq c_0\ \p\left(\max_{k=1,\ldots,n}W_{(k/n)^{2H}}\leq -\sqrt{2\ln(n)}+ \frac{\ln(\ln(n))+\ln(4\pi)}{2\sqrt{2\ln(n)}}\right)
        \\
        &\geq c_0\ \p\left(\max_{k=1,\ldots,n}W_{(k/n)^{2H}}-W_{(1/n)^{2H}}\leq 1, W_{(1/n)^{2H}}\leq -\sqrt{2\ln(n)}-1\right) \\
        &= c_0\  \p\left(\max_{k=1,\ldots,n}W_{(k/n)^{2H}}-W_{n^{-2H}}\leq 1\right) \cdot  \p\left( W_{n^{-2H}}\leq -\sqrt{2\ln(n)}-1\right)  \\
        &\geq c_0\  \p\left(\max_{t\in [0, 1-n^{-2H}]}W_t\leq 1\right)\cdot  \p\left( W_{1}\leq \frac{-\sqrt{2\ln(n)}-1}{n^{-H}}\right)\\
        &\geq c_0\ \p\left(\max_{t\in [0, 1]}W_t\leq 1\right)\cdot  \p\left( W_{1}\leq \frac{-\sqrt{8\ln(n)}}{n^{-H}}\right),
    \end{align*}
    where we use the independent increments (step 3) and the stationary increments (step 4). Furthermore, in the last step we used that for $n$ large enough $\sqrt{2\ln(n)} \geq 1$. The first probability is a constant greater than zero and does not depend on $n$ or $H$. The second probability can be bounded from below using (\ref{eqn:Mills_Ratio_Inequalities}) and using that $H\leq \frac{\delta \ln(\ln(n))}{\ln(n)}$ which is equivalent to $n^{H}\leq \ln(n)^{\delta}$ so that
    \begin{align*}
        \p\left(\max_{k=1,\ldots,n}B_k^H \leq 0\right) &\geq  c_1 \p\left( W_{1}\leq -n^H\sqrt{8\ln(n)}\right) \\
     %   &\geq c_1\,\frac{1}{\sqrt{8\ln(n)}n^H+(\sqrt{8\ln(n)}n^H)^{-1}}e^{-4\ln(n)n^{2H}} \\
        &\geq c_1 \,\frac{1}{2\sqrt{8\ln(n)}n^H}e^{-4\ln(n)n^{2H}} \\
        &\geq \frac{c_1}{2\sqrt{8}}(\ln(n))^{-(1/2+\delta)}e^{-4\ln(n)^{1+2\delta}}\\
        &\geq c_2 e^{-c_3\ln(n)^{1+2\delta}},
    \end{align*}
    for some $c_1,c_2,c_3 > 0$,
    where we used that $\sqrt{8\ln(n)}n^H \geq1 \geq  (\sqrt{8\ln(n)}n^H)^{-1}$ and that $(\ln(n))^{-(1/2+\delta)}\geq e^{-\ln(n)^{1+2\delta}}$ for large enough $n$, independently of $H$.
\end{proof}

\subsection{A uniform upper bound for the persistence probability}

We now have everything in place to prove the next proposition, which provides an explicit upper bound for the persistence probability of an FBM with fixed Hurst exponent up to time $n\in\mathbb{N}$. The bound is stated in a form that makes the dependence on $H$ transparent. Moreover, it contains an additional discretizations parameter $m$ that reflects the sampling scheme used in the proof and will later be chosen as a function of $H$ to prevent the estimate from deteriorating as $H\downarrow 0$.

The key step in proving this upper bound is to compare the persistence probability below the level \(1\) with a corresponding persistence probability below the level \(-1\) for a discretely sampled version of the process. The latter can be controlled by Corollary~\ref{Cor:No_Cond_on_H}. To pass from the level \(-1\) to the level \(1\), we use a change-of-measure argument for the resulting Gaussian vector. The lower bounds from Lemmas~\ref{Lem:Case1} and~\ref{Lem:Case2} will then ensure that the cost of this change of measure can be controlled uniformly in the relevant regimes of \(H\).

\begin{Proposition}\label{Prop:No_cond_on_H}
    Let $\delta > 1/4$. There exist some constants $a, b > 0$ and some $n_0 \in \mathbb{N}$ such that for all $m \in \mathbb{N}$ and all $n \geq n_0$
    \begin{align}
        \p\left(\max_{t\in[0,n]}B_t^H \leq 1\right)\leq \frac{an^{H-1}}{m\sqrt{H}}e^{m^{2H}b\ln(nm)^{1/2+\delta}},  \label{No_cond_on_H}  
    \end{align}
    where $(B_t^H)_{t\geq 0}$ is a fractional Brownian motion with Hurst exponent  $H\in (0,1)$.
\end{Proposition}

\begin{proof}
We start by noting that
    \begin{align}
        \p\left(\max_{t\in[0,n]}B_t^H \leq 1\right) \leq \p\left(\max_{k=1,\ldots,nm}B_{k/m}^H\leq 1\right). \label{pos}
    \end{align}
    We can proceed by
    applying Corollary \ref{Cor:No_Cond_on_H} and Lemma \ref{Max_bound} to obtain
    \begin{align}
        \p\left(\max_{k=1,\ldots,nm}B_{k/m}^H \leq -1 \right) &\leq 2n^{H-1}\frac{\E\left[ \max_{t\in [0,1]}B_t^H\right]}{m} % \notag\\
        \leq n^{H-1}\,\frac{c_0}{m\sqrt{H}}, \label{neg}
    \end{align}
    for some constant $c_0 > 0$ independent of $H$. In order to compare (\ref{pos}) (note: boundary $+1$) and the probability in (\ref{neg}) (note: boundary $-1$), we want to apply Proposition~3 in \cite{Au18} (see Proposition~1.6 in \cite{aurzadaDereich2013AIHP427} for the original statement, which contains a small error, and the proof) to the discretely sampled fractional Brownian motion
\[
X^{H}_{n,m}:=\big(B_t^H\big)_{t\in T_{n,m}},\qquad 
T_{n,m}:=\Big\{\frac{k}{m}: k=1,\dots,nm\Big\}.
\]
This is a centered Gaussian vector in $\R^{|T_{n,m}|}$ with covariance kernel
\[
K^H(s,t):=\E[B_s^H B_t^H],\qquad s,t\in T_{n,m}.
\]
Let $\mathscr{H}_{n,m}$ denote the reproducing kernel Hilbert space (RKHS) associated with
$\big(B_t^H\big)_{t\in T_{n,m}}$, i.e.\ the completion of the linear span of $\{K^H(t,\cdot): t\in T_{n,m}\}$
equipped with the inner product
\[
\langle K^H(s,\cdot),K^H(t,\cdot)\rangle_{\mathscr H^{H}_{n,m}} = K^H(s,t),
\qquad s,t\in T_{n,m}.
\]
We set
\[
\kappa_{n,m}:=\inf_{t\in T_{n,m}} K^H(1/m,t).
\]
Since $t\ge m^{-1}$ for all $t\in T_{n,m}$, we have $K^H(1/m,t)\ge m^{-2H}/2$ and hence $\kappa_{n,m}\ge m^{-2H}/2$.
Define
\[
f^{H}_{n,m}(t):=\frac{2}{\kappa_{n,m}}\,K^H(1/m,t),\qquad t\in T_{n,m}.
\]
Then $f^{H}_{n,m}\in\mathscr H^{H}_{n,m}$ and $f^{H}_{n,m}(t)\ge 2$ for all $t\in T_{n,m}$.
Moreover,
\[
\|f^{H}_{n,m}\|^2_{\mathscr H^{H}_{n,m}}
=\Big(\frac{2}{\kappa_{n,m}}\Big)^2\langle K^H(1/m,\cdot),K^H(1/m,\cdot)\rangle_{\mathscr H^{H}_{n,m}}
=\Big(\frac{2}{\kappa_{n,m}}\Big)^2 K^H(1/m,1/m)
\le 16m^{2H}.
\]
This bound is a crucial uniform estimate in the change-of-measure argument.
    Using Proposition 3 in \cite{Au18} (in the third step), we can thus conclude that
    \begin{align}
        & \p\left(\max_{k=1,\ldots,nm}B_{k/m}^H \leq -1\right)  \notag
        \\
        &= \p(\forall t \in T_{n,m}: B_t^H + f_{n,m}^H(t) \leq -1 + f_{n,m}^H(t))\notag\\
        &\geq \p(\forall t \in T_{n,m}: B_t^H + f_{n,m}^H(t) \leq 1) \notag\\
        &\geq \p\left(\max_{k=1,\ldots,nm}B_{k/m}^H \leq 1\right)  \cdot\exp\left(-\sqrt{2\|f_{n,m}^H\|^2_{\mathscr{H}_{n,m}^H}\ln(1/p_{n,m})}-\|f_{n,m}^H\|_{\mathscr{H}_{n,m}^H}^2/2\right) \notag\\
        &\geq \p\left(\max_{k=1,\ldots,nm}B_{k/m}^H \leq 1\right) \cdot\exp\left(-m^{H}\sqrt{32\ln(1/p_{n,m})}-8m^{2H}\right), \label{neg2}
    \end{align}
    where
    \begin{align*}
        p_{n,m}:= \p\left(\max_{k=1,\ldots,nm}B_{k/m}^H \leq 1\right).
    \end{align*}
    If we have a suitable lower bound for $p_{n,m}$ we are finished with connecting (\ref{pos}) and (\ref{neg}). This is where Lemma~\ref{Lem:Case1} and Lemma~\ref{Lem:Case2} come into play. In order to use them, we will make a case distinction. However, we are not exactly in the setting of the two lemmas yet, which is why we use
    \begin{align*}
        p_{n,m}&=\p\left(\max_{k=1,\ldots,nm}B_{k/m}^H \leq 1\right) 
        \geq \p\left(\max_{j=1,\ldots,nm}B_{j/m}^H \leq 0\right)
        = \p\left(\max_{j=1,\ldots,nm}B_{j}^H \leq 0\right),
    \end{align*}
    where we used self-similarity in the last step. We are now finally able to apply Lemma \ref{Lem:Case1} and Lemma \ref{Lem:Case2} for the number of points $nm$ as follows. \\[0.3em]
    \textbf{Case 1: $H\geq \frac{\delta\ln(\ln(nm))}{\ln(nm)}$: } Since $\delta > 1/4$, we can apply Lemma \ref{Lem:Case1} which gives us
    \begin{align*}
        \p\left(\max_{j=1,2,\ldots,nm}B_j^H \leq 0\right) \geq \frac{c_1}{(nm)^3},
    \end{align*}
    for some constant $c_1 > 0$ and large enough $n$.\\[0.3em]
    \textbf{Case 2: $H < \frac{\delta \ln(\ln(nm))}{\ln(nm)}$:} Lemma \ref{Lem:Case2} gives us
    \begin{align*}
        \p\left(\max_{j=1,2,\ldots,nm}B_j^H \leq 0\right) \geq  c_2e^{-c_3\ln(nm)^{1+2\delta}},
    \end{align*}
    for some constants $c_2, c_3 > 0$ and large enough $n$.
    Let us furthermore define
    \begin{align*}
        \Omega:= \left\{H \geq \frac{\delta\ln(\ln(nm))}{\ln(nm)}\right\}.
    \end{align*}
    Putting together (\ref{neg}) and (\ref{neg2}), we can estimate
    \begin{align*}
        \p&\left(\max_{k=1,\ldots,nm}B_{k/m}^H\leq 1\right) \\
        &\quad = \p\left(\max_{k=1,\ldots,nm}B_{k/m}^H\leq 1\right) \1_{\Omega}+ \p\left(\max_{k=1,\ldots,nm}B_{k/m}^H\leq 1\right) \1_{\Omega^c}\\
        & \quad \leq \frac{c_0n^{H-1}}{m\sqrt{H}}\bigg(\1_\Omega\cdot \exp\left(m^{H}\sqrt{32\cdot 3 \cdot \ln(mn)-32\ln( c_1)}+8m^{2H}\right)\\
        &\quad \quad \quad +  \1_{\Omega^c}\cdot\exp\left(m^H\sqrt{32\cdot c_3\ln(mn)^{1+2\delta}-32\cdot \ln(c_2)}+8m^{2H}\right)\bigg) \\
        &\quad \leq \frac{c_{4}n^{H-1}}{m\sqrt{H}}e^{m^Hc_{5}\ln(mn)^{1/2+\delta}+c_6m^{2H}}\\
        &\quad \leq \frac{c_{4}n^{H-1}}{m\sqrt{H}}e^{m^{2H}c_{7}\ln(mn)^{1/2+\delta}},
    \end{align*}
    for some constants $c_4, c_5,c_6, c_7 > 0$ for large enough $n$, where we also used that $m^H \leq m^{2H}$. Combining this chain of inequalities with (\ref{pos}) shows the claim.
\end{proof}

%%%%%%%%%%%%%%%%%%%%%%%%%%%%%%%%%%%%%%%%%%%%%%%%%%%%%%%%%%
%%%%%%%%%%%%%%%%%%%%%%%%%%%%%%%%%%%%%%%%%%%%%%%%%%%%%%%%%%
%%%%%%%%%%%%%%%%%%%%%%%%%%%%%%%%%%%%%%%%%%%%%%%%%%%%%%%%%%
\subsection{Upper Bound}\label{sec_proof_of_UB}

With Proposition \ref{Prop:No_cond_on_H} at hand, we have everything in place to prove the upper bound in Theorem~\ref{Thm}. 

\begin{proof}[Proof of the upper bound in Theorem~\ref{Thm}] 
    The main idea is to apply Proposition \ref{Prop:No_cond_on_H} for some $\delta \in (1/4, 1/2)$ for any fixed $H\in (0,H_0]\cap (0,1)$  for $n := \lfloor T \rfloor$, with an appropriate choice of $m\in \mathbb{N}$ that may depend on $H$. Since $\frac{1}{\sqrt{H}}$ explodes as $H\rightarrow 0$, we will choose $m$ in such a way, that the r.h.s.\ of (\ref{No_cond_on_H}) does not explode. To do so, we make a case distinction on $H$ into a case where $H$ is bounded away from $0$ and the case where it is not. Let $c_0 > 0$ be a constant that will be chosen later.\\[0.3em]
    \textbf{Case 1:  $H> c_0$:} We can apply Proposition \ref{Prop:No_cond_on_H} with $m = 1$ to obtain
    \begin{align*}
        \p_B\left(\max_{t\in[0,T]}B_t^H \leq 1\right) &\leq \p_B\left(\max_{t\in[0,n]}B_t^H \leq 1\right)  \leq \frac{c_1n^{H-1}}{\sqrt{H}}e^{c_2\ln(n)^{1/2+\delta}}\leq \frac{c_1n^{H_0-1}}{\sqrt{c_0}}e^{c_2\ln(n)^{1/2+\delta}},
    \end{align*}
    for some constants $c_1, c_2 > 0$ that do not depend on $H$ or $m$,
    where we used that $H \leq H_0$.\\[0.3em]
    \textbf{Case 2: $H \leq c_0$:} We set $m:= \lceil \frac{1}{\sqrt{H}}\rceil^2$. With that 
    choice, we again apply Proposition \ref{Prop:No_cond_on_H} to estimate
    \begin{align}
        \p_B\left(\max_{t\in[0,T]}B_t^H \leq 1\right) &\leq \p_B\left(\max_{t\in[0,n]}B_t^H \leq 1\right) \notag\\ \notag
        &\leq \frac{c_1n^{H-1}}{m\sqrt{H}}e^{m^{2H}c_2\ln(nm)^{1/2+\delta}}\\ 
        &\leq \frac{c_1 n^{H_0-1}}{\sqrt{m}}e^{m^{2H}c_2\ln(nm)^{1/2+\delta}}\notag
    \end{align}
    where we used $\frac{1}{\sqrt{m}\sqrt{H}}\leq 1$ and $H \leq H_0$ in the last step. Since $\sqrt{m} = \lceil \frac{1}{\sqrt{H}}\rceil \leq \frac{1}{\sqrt{H}}+1 \leq \frac{2}{\sqrt{H}}$ we have 
    \begin{align*}
        m^{2H} \leq \left(\frac{4}{H}\right)^{2H} \leq 16 e^{2/e} =: c_3,
    \end{align*}
    for any $H \in (0,1)$.
    That means we can estimate 
    \begin{align}
         \p_B\left(\max_{t\in[0,T]}B_t^H \leq 1\right) &\leq \frac{c_1n^{H_0-1}}{\sqrt{m}}e^{c_2m^{2H}\ln(nm)^{1/2+\delta}} \notag \\ \notag
         &\leq \frac{c_1n^{H_0-1}}{\sqrt{m}}e^{c_2c_3\ln(nm)^{1/2+\delta}} \\ 
         &= \frac{c_1n^{H_0-1}}{\sqrt{m}}e^{c_4\ln(nm)^{1/2+\delta}}\label{stern}
    \end{align}
    for some constant $c_4 > 0$.
    Because of $1/2+\delta < 1$, there exists an $m_0 \in \mathbb{N}$ such that
    \begin{align*}
        c_4(2\ln(m))^{1/2+\delta}-\frac{1}{2}\ln(m) \leq 0 \quad \text{for all $m \geq m_0$}.
    \end{align*}
    Set $c_0:=1/m_0$. Then $H\leq c_0$ implies $m\geq \frac{1}{H} \geq m_0$. For such $m$, we further estimate
    \begin{align*}
        \frac{1}{\sqrt{m}}e^{c_4\ln(nm)^{1/2+\delta}} &= e^{c_4(\ln(n)+\ln(m))^{1/2+\delta}-\frac{1}{2}\ln(m)} \\ 
        &\leq \1_{\{n \geq m\}}e^{c_4(2\ln(n))^{1/2+\delta}}+\1_{\{n < m\}} e^{c_4(2\ln(m))^{1/2+\delta}-\frac{1}{2}\ln(m)}\\ 
        &\leq \1_{\{n \geq m\}}e^{c_4(2\ln(n))^{1/2+\delta}}+\1_{\{n < m\}}\\  \notag
        &\leq e^{c_{5}\ln(n)^{1/2+\delta}},  
    \end{align*}
    for some constant $c_{5} > 0$ that does not depend on $H$. Inserting this into (\ref{stern}) yields
    \begin{align*}
         \p_B\left(\max_{t\in[0,T]}B_t^H \leq 1\right) &\leq c_1n^{H_0-1}e^{c_{5}\ln(n)^{1/2+\delta}}.
    \end{align*}
    Piecing the two cases together gives us the following inequality
    \begin{align*}
        \p_B\left(\max_{t\in[0,T]}B_t^H \leq 1\right) &\leq \1_{\{H > c_0\}}  \frac{c_1n^{H_0-1}}{\sqrt{c_0}}e^{c_2\ln(n)^{1/2+\delta}}
         + \1_{\{H \leq c_0\}} c_1n^{H_0-1}e^{c_5\ln(n)^{1/2+\delta}}\\*
        &\leq c_6n^{H_0-1}e^{c_7\ln(n)^{1/2+\delta}},
    \end{align*}
    for some constants $c_6, c_7 > 0$ independent of $H$.
    Since the r.h.s.\ of the inequality above does not depend on $H$ anymore, we can use this estimate to obtain our upper bound of the persistence probability of interest:
    \begin{align*}
        \E_\mathcal{H}\left[\p_B\left(\max_{t\in[0,T]}B_t^\mathcal{H} \leq 1\right)\right] \leq c_6n^{H_0-1}e^{c_7\ln(n)^{1/2+\delta}}.
    \end{align*}
    We proceed using $n^{H_0-1} = O(T^{H_0-1})$ and $c_6e^{c_{7}\ln(n)^{1/2+\delta}} = O(T^{o(1)})$ because of $n = \lfloor T \rfloor\sim T$ for $T \rightarrow \infty$ and $\delta < 1/2$, which implies the upper bound in Theorem \ref{Thm}.
\end{proof}

\section{Summary}
We have studied the persistence probability of fractional Brownian motion with random Hurst exponent, as introduced in \cite{balcerek2022_FBMre_chaos}. In this model, the Hurst parameter is sampled once and then kept fixed along the whole trajectory.
If \(H_0\) denotes the essential supremum of the distribution of the random Hurst exponent then the persistence probability decays as
\[
    \E_\mathcal{H}\left[\p_B\left(\max_{t\in [0,T]}B_t^\mathcal{H} \leq 1\right)\right]=T^{-(1-H_0)+o(1)} .
\]
Thus the annealed persistence exponent is governed by the upper edge of the distribution of \(\mathcal{H}\).

We also obtained the corresponding small-barrier asymptotics on a fixed time interval. While the small-barrier result follows from monotonicity in \(H\) of the relevant probability, the long-time persistence problem requires uniform estimates that remain stable as \(H\downarrow0\). This is the main technical difficulty addressed in this paper.

%%%%%%%%%%%%%%%%%%%%%%%%%%%%%%%%%%%%%%%%%%%%%%%%%%%%%%%%%%
%%%%%%%%%%%%%%%%%%%%%%%%%%%%%%%%%%%%%%%%%%%%%%%%%%%%%%%%%%
%%%%%%%%%%%%%%%%%%%%%%%%%%%%%%%%%%%%%%%%%%%%%%%%%%%%%%%%%%

\normalem

\bibliographystyle{alpha}
\newcommand{\etalchar}[1]{$^{#1}$}

\appendix

\end{document}